\documentclass[11pt, a4paper]{amsart}
\usepackage{amsfonts,amssymb,amsmath, amsthm}
\usepackage{mathrsfs, mathtools}
\usepackage{lmodern}
\usepackage{qsymbols}
\usepackage{latexsym}
\usepackage[noadjust]{cite}
\usepackage{pdfsync}
\usepackage{bm}
\usepackage{enumitem}

\newtheorem{thm}{Theorem}[section]
\newtheorem{lem}[thm]{Lemma}
\newtheorem{prop}[thm]{Proposition}
\newtheorem{cor}[thm]{Corollary}

\theoremstyle{definition}
\newtheorem{defn}[thm]{Definition}
\newtheorem{rem}[thm]{Remark}

\numberwithin{equation}{section}
\usepackage[plainpages=false,pdfpagelabels,backref=page,citecolor=red]{hyperref}
\usepackage{xcolor}
\hypersetup{
colorlinks,
linkcolor={cyan!90!black},
citecolor={magenta},
urlcolor={green!40!black}
} %important to load after amsthm and before the rest
\renewcommand{\Cap}{\text{cap}}
\newcommand{\tor}{\textup{tor}}
\newcommand{\R}{\mathbb{R}}

\newcommand{\diam}{\operatorname{diam}}
\newcommand{\dist}{\operatorname{dist}}
\newcommand{\loc}{\operatorname{loc}}
\renewcommand{\L}{\operatorname{L}} %Lebesgue spaces

\renewcommand{\d}{\, \mathrm{d}} %differential
\renewcommand{\limsup}
{\textup{limsup}}
\renewcommand{\liminf}
{\textup{liminf}}
\def\Xint#1{\mathchoice
{\XXint\dipslaystyle\textstyle{#1}}%
{\XXint\textstyle\scriptstyle{#1}}%
{\XXint\scriptstyle\scriptscriptstyle{#1}}%
{\XXint\scriptscriptstyle%
\scriptscriptstyle{#1}}%
\!\int}
\def\XXint#1#2#3{{\setbox0=\hbox{$#1{#2#3}{%
\int}$ }
\vcenter{\hbox{$#2#3$ }}\kern-.6\wd0}}
\def\overlineint{\,\Xint -} % \, corrects the \! used in the definition
\def\overlineiint{\overlineint_{} \kern-.4em \overlineint}
\def\overlineiiint{\overlineiint_{} \kern-.4em \overlineint}
\renewcommand{\iint}{\int_{}\kern-.34em \int} %\, minor space between the integrals
\renewcommand{\iiint}{\iint_{}\kern-.34em \int} %\, minor space between the integrals
\newcommand{\note}[1]{{\color{black} #1}}

\title{Boundary behavior of solutions to fractional $p$-Laplacian equation}

\author{Alireza Ataei}
\email{alireza.ataei@math.uu.se}
\address{Department of Mathematics, Uppsala University, S-751 06 Uppsala,
Sweden}



\thanks{}

\date{\today}

\begin{document}

\maketitle
%\tableofcontents

%%%%%%%%%%%%%%%%%%%%%%%%%%%%%%%%%%%%%%%%%%%%%%%%%%%%%%%%%%%%%%%%%%%%%%%%%%%%
%abstract 
\begin{abstract}
  In this work, a generalized Hopf's lemma and a global boundary Harnack inequality are proved for solutions to fractional $p$-Laplacian equations. Then, the isolation of the first $(s,p)$-eigenvalue is shown in bounded open sets satisfying the Wiener criterion. 
\end{abstract}

%%%%%%%%%%%%%%%%%%%%%%%%%%%%%%%%%%%%%%%%%%%%%%%%%%%%%%%%%%%%%%%%%%%%%%%%%%%%
\section{Introduction}
We maintain the previous work in \cite{AA} and prove the boundary properties of solutions to fractional $p$-Laplacian equations.

The first result is a generalized Hopf's Lemma. To bring the result, we need the notion of Wiener regular boundaries, $\delta$-neighbourhoods, and the torsion function. Let $\Omega \subset \R^n$ be a bounded open set, $p>1$, and $0<s<1$.  Denote $(-\Delta_p)^s$ as the $s$-fractional $p$-Laplacian, which satisfies \begin{align*}
  (-\Delta_p)^s u(x) = 2\lim_{\epsilon \to 0}  \int_{\R^n \setminus B(x,\epsilon)} \frac{|u(x)-u(y)|^{p-2} (u(x)-u(y))}{|x-y|^{n+ps}} \, \d y,
\end{align*} pointwise for $x \in \R^n$. We say that $\Omega$ has regular boundary for the $s$-fractional $p$-Laplacian if for every $f \in \L^{\infty}(\Omega),g \in C(\R^n)$, and every weak solution $u$ of $(-\Delta_p)^s u = f$ in $\Omega$ with $u=g$ in $\R^n \setminus \Omega$, we have $u \in C(\R^n),$ see Section \ref{sec:wiener} for more details. Now, assume that $\Omega \subset \R^n$ is a bounded open set, which has Wiener regular boundary for the
$s$-fractional $p$-Laplacian. For $\delta>0$, the $\delta$-nieghbourhood of $\Omega$, denoted by $\Omega_{\delta}$, is defined by $\{x \in \R^n: \dist(x,\overline \Omega)< \delta\}.$ The torsion function $u_{\tor} \in  \L^{p-1}_{ps}(\R^n) \cap C(\R^n)$ satisfies
\begin{align*}
u_{\tor} = 0, \quad &\text{in } \R^n \setminus \Omega,\\
    (-\Delta_p)^{s} u_{\tor} = 1, \quad &\text{in } \Omega,
\end{align*}
in the viscosity sense, see Proposition \ref{prop:existence} and Proposition \ref{prop:equivalenceweakvisco} for the existence of $u_{\tor}$. We say that $K \Subset \Omega$ if $\overline{K} \subset \Omega.$

\begin{lem}
 \label{lem:Hopf's lemma}
    Let $u \in \L^{p-1}_{ps}(\R^n) \cap C(\overline \Omega_{\delta})$ be a non-negative function for a $\delta >0$ and $K \Subset \Omega$. Assume that $(-\Delta_p)^s u \geq f$ in $\Omega$ in the viscosity sense, where $f \in C(\Omega)$ satisfies 
\begin{align*}
 f(x_0) > - 2 \int_{\R^n} \frac{u^{p-1}(y)}{|x_0-y|^{n+ps}} \, \d y, \quad &\text{if } x_0 \in \Omega, \, u(x_0)=0,\\
 \underset{\Omega \ni x \to x_0}{\limsup} f(x) \geq - 2 \int_{K} \frac{u^{p-1}(y)}{|x_0 - y|^{n+ps}} \, \d y, \quad & \text{if } x_0 \in \partial \Omega, u(x_0) =0.
\end{align*}
Then, $u>0$ in $\Omega$ and
\begin{align*}
   u \geq C u_{\tor}, \quad \text{in } \Omega,
\end{align*}
    for a constant $C>0.$

\end{lem}
Setting $f = g(u),$ we arrive at the following result:
\begin{cor}
\label{cor:key}
     Let $u \in \L^{p-1}_{ps}(\R^n) \cap C(\overline \Omega_{\delta})$ be a non-negative viscosity supersolution of \note{$(-\Delta_p)^s u = g(u)$} in $\Omega$, where $\delta>0, g \in C([0,\infty)), g(0) =0.$ Then, either $u=0$ a.e. in $\R^n$, or $u>0$ in $\Omega$ and $u \geq C u_{\tor}$ for a constant $C>0.$
\end{cor}
If $\Omega$ has a $C^{1,1}$ boundary, then, by \cite[Lem. 2.3]{IMS1}, $u_{\tor}(x) \geq C \dist(x,\partial \Omega)^s$ for $x \in \Omega$, where $C>0$ is a constant. Hence, the above result generalizes the previous versions of Hopf's lemma for the $s$-fractional $p$-Laplacian in \cite{CRS,GS,R,DQ,LZ,IMS2}.

We remark that we could not verify the argument in \cite[Lem. 4.1]{DQ}, which considers the case that $g(u) = c |u|^{p-2} u$ in a ball \note{$B \subset \R^n$} of radius $R,$ where $c \in C (\overline{\Omega} )$ is negative. To elaborate on the issue, the authors take the set $B_{\rho} :=\{x \in B: \dist(x,\partial B) < \rho\}$ and a compact \note{subset} $K \subset B \setminus B_{\rho}$, where $\rho$ is taken small enough such that $(-\Delta_p)^s \dist(x,\partial \Omega)^s \in \L^{\infty}(B_{\rho})$, see \cite[Thm. 2.3]{IMS}. \note{They choose} $\alpha$ large enough such that $(-\Delta_p)^s (d^s + \alpha 1_{K})$ \note{becomes} very small in compare to $c |u|^{p-1} u$ in $B_{\rho}$, where $d$ is the function $\dist(x,\partial B)$. Then, they consider $0<\epsilon<1$ small enough such that $\epsilon (R^s + \alpha) \leq \inf_{B \setminus B_{\rho}} u$ and define $v:= \epsilon (d^s + \alpha 1_{K}) $. Finally, they claim that, since $v \leq u$ in $\R^n \setminus B_{\rho}$ and $(-\Delta_p)^s v \leq (-\Delta_p)^s u $ in $B_{\rho}$, one can apply comparison principle to obtain $v \leq u$ in $B_{\rho}$. However, $(-\Delta_p)^s v = \epsilon^{p-1} (-\Delta_p)^s (d^s + \alpha 1_{K})$ and \note{the decrease of} $\epsilon>0$ \note{increases} the value of $(-\Delta_p)^s v$, since \note{$d^s + \alpha 1_{K} \leq c |u|^{p-1} u \leq 0$ by the maximum principle, see \cite[Thm. 1.2]{IMS}.} Hence, it is not clear that $(-\Delta_p)^s v$ remains below $(-\Delta_p)^s u$ in $B_{\rho}$.
 
Note that as it is mentioned in \cite{AA} and \cite[Rem. 2.8]{IMS2}, Corollary \ref{cor:key} does not hold for the local $p-$Laplacian, see \cite{PS,V} for the necessary assumptions on $g$ to have strong maximum property. Hence, the nonlocal property of fractional $p-$Laplacian plays a key role in the proof.

We observe that unlike \cite[Lem. 1.2]{AA}, to prove Lemma \ref{lem:Hopf's lemma}, we need a stronger continuity of $u$ around a neighborhood of $\overline \Omega$.

The second result is a global boundary Harnack theorem. We briefly mention that $V_g^{s,p}(\Omega|\R^n)$ is the fractional Sobolev space on $\R^n$ with the boundary value $g$ in the trace space $V^{s,p}(\Omega|\R^n)$, see Section \ref{sec:spaces} for more details. 
\begin{thm}
\label{thm:boundaryharnack}
    Let $\delta>0$, $u \in C(\overline \Omega_{\delta}) \cap V_{g_u}^{s,p}(\Omega|\R^n),v \in C(\overline \Omega_{\delta}) \cap V_{g_v}^{s,p}(\Omega|\R^n)$ satisfy  
    \begin{align*}
    u>0 , \,v>0 \quad & \text{in }  \Omega,\\
   0 \leq \frac{1}{B} g_v \leq g_u \leq B g_v \leq M, \quad &\text{in } \R^n \setminus \Omega,
    \end{align*}
  for $B>0, M \geq 0$, and
    \begin{align*}
   - 2 (\mathrm{diam}\, \Omega)^{-(n+ps)}\int_{K} u^{p-1}(y)\, \d y \leq (-\Delta_p)^s u \leq  1, \quad & \text{in } \Omega,\\
    -2 (\mathrm{diam}\, \Omega)^{-(n+ps)}\int_{K} v^{p-1}(y) \, \d y  \leq   (-\Delta_p)^s v \leq 1, \quad &\text{ in } \Omega,
    \end{align*}
    in the locally weak sense, where $K \Subset \Omega$. If either $u(x_0)\geq D , \; v(x_0)\geq D$ or $\|u\|_{\L^q(\Omega \setminus K)} \geq D , \;  \|v\|_{\L^q(\Omega \setminus K)}\geq D$ \note{for a fixed point $x_0 \in \Omega \setminus K$ and some constants} $D>0,1 \leq q  <\infty$, then
    \begin{align*}
       C_1 \leq \frac{u}{v} \leq C_2, \quad \text{in } \Omega,
    \end{align*} where $C_1, C_2$ are positive constants depending on $\Omega, \delta, K, n, s,p, D,B,M, x_0$ or $q$.
\end{thm}
Up to the knowledge of the author, there are several results for the boundary Harnack theorem for the linear fractional Laplacian, see \cite{BFV,B,RS,AA,RT}, but there is none for the fractional $p-$Laplacian. \note{In the case of $g_u=0,v=u_{\tor}$ in Theorem \ref{thm:boundaryharnack}, we derive the following corollary:
\begin{cor}
     Let $\delta>0$, $u \in C(\overline \Omega_{\delta}) \cap V_{0}^{s,p}(\Omega|\R^n)$ satisfy $ u>0$ in  $\Omega$ and 
    \begin{align*}
    - 2 (\mathrm{diam}\, \Omega)^{-(n+ps)}\int_{K} u^{p-1}(y)\, \d y \leq (-\Delta_p)^s u \leq  1, \quad  \text{in } \Omega,
    \end{align*}
    in the locally weak sense, where $K \Subset \Omega$. If either $u(x_0)\geq D$ or $\|u\|_{\L^q(\Omega \setminus K)} \geq D$, where $x_0 \in \Omega \setminus K$ is a fixed point and $D>0,1 \leq q  <\infty$ are fixed constants, then
    \begin{align*}
      \sup_{B} u \leq C \inf_{B} u,
    \end{align*} for every subset $B \Subset \Omega$, where $C$ is a positive constant, which depends on $\Omega, \delta, K, n, s,p, D, x_0$ or $q$.
\end{cor}

}

\note{In the last result}, we prove the isolation of the first $(s,p)$-eigenvalue.

Define the Sobolev exponent
\begin{align*}
    p^{\ast}_s := \begin{cases}
        \frac{pn}{n-ps}, \quad &\text{if } ps < n,\\
        \infty, \quad &\text{if } ps \geq n.
    \end{cases}
\end{align*}
and consider the following minimization problem
\begin{equation}\label{eq:fist-eignevalue}
\begin{aligned}
\Lambda_{p,q}:=\inf_{\phi \in C^{\infty}_0(\Omega)} \biggl\lbrace \iint_{\R^n \times \R^n} \frac{|\phi(x)-\phi(y)|^p }{|x-y|^{n+ps}} \, \d x \d y \, : \|\phi\|_{\L^q(\Omega)}=1 \biggr\rbrace,
\end{aligned}
\end{equation}
for $1<q <p^{\ast}_s$. 
\begin{thm}
\label{thm:isolation}
If $1<q \leq p$ and $\Omega$ has regular boundary for the $s$-fractional $p$-Laplacian, then there exist no sequences $\note{\lambda_i>\Lambda_{p,q}}, u_i \in V^{s,p}_0(\Omega|\R^n)$, \note{which satisfy} 
\begin{align*}
\|u_i\|_{\L^q(\Omega)}&=1,\\
\lim_{i \to \infty} \lambda_i &= \Lambda_{p,q},\\
(-\Delta_p)^s u_i &= \lambda_i |u_i|^{q-2}u_i, \quad \text{weakly in } \Omega.
\end{align*}
    
\end{thm}
This generalizes previous results in \cite{AA,BP,FL}. Moreover, we bring a shorter proof for the case of $C ^{1,1}$ boundary condition \note{and $p \geq 2$}, see Remark \ref{rem:shortproof}. 
\note{We remark that the range $1<q \leq p$ is used in two major parts, which play key roles in the proof of Theorem \ref{thm:isolation}. First, we derive that the minimizes of $\Lambda_{p,q}$ are unique, up to a multiplication by a constant, and strictly positive or negative on $\Omega$, see Proposition \ref{prop:comparison}. Second, we imply that the only $(s,p)$-eigenvalue $\lambda>0$ with non-negative $(s,p)$-eigenfunction $u$, i.e., a non-negative function $u \in V^{s,p}_0(\Omega|\R^n)$ which satisfies $(-\Delta_p)^s u = \lambda |u|^{q-2}u$ weakly in $\Omega$ and $\|u\|_{\L^q(\Omega)}=1$, is $\Lambda_{p,q}$, see Proposition \ref{prop:signchanginh}.} Notice that in the case \note{of} $q=p$ in \cite{BP}, the isolation of the first $(s,p)$-eigenvalue is proved without the assumption of Wiener regularity on the boundary of $\Omega.$

\note{The main difficulty in this work is the proof of} a similar version of \cite[Lem. 5.1]{AA} for solutions to fractional $p-$Laplacian. The nonlinearity of the equation causes the argument in \cite[Lem. 5.1]{AA} to fall down. To resolve the issue, we apply the method of the proof for \note{the} comparison principle for viscosity solutions. Then, the argument follows in the same direction as proof of \cite[Lem. 3.1]{GS}. One of the interesting aspects of this work is the simplicity of \note{proofs}, although one works with the weakest notion of solutions, namely the viscosity solutions. 

The summary of the whole work is as follows: First, we bring the definitions of spaces and notions of solutions in Section \ref{sec:Preliminaries}. Then, in Section \ref{sec:wiener}, we present the Wiener criterion for the $s$-fractional $p$-Laplacian with a nonzero right-hand side. Afterward, the first two results, namely a generalized Hopf's lemma and global boundary Harnack inequality are proved in Section \ref{sec:HopfandHarnack}. Finally, in Section \ref{sec:isolation}, we prove the isolation of the first $(s,p)$-eigenvalue.

\section{Preliminaries}
\label{sec:Preliminaries}
In the entire work, $p>1$, $0<s<1$, and $\Omega \subset \R^n$ is a bounded open set.
\subsection{Spaces}
\label{sec:spaces}
For every $K \subset \R^n$ and $x \in \R^n$, we define $K+x :=\{y \in \R^n: y= z+x, z \in K\}.$
Let $\L^{p-1}_{\loc}(\R^n)$ be the space of measurable functions $u: \R^n \to \R$ such that $\|u\|_{\L^{p-1}(K)}< \infty$ for every compact \note{subset} $K \subset \R^n.$ The space $\L^{p-1}_{ps}(\R^n)$ consists of $u \in \L^{p-1}_{\loc}(\R^n)$, satisfying
\begin{align*}
   \note{ \|u\|^{p-1}_{\L^{p-1}_{ps}(\R^n)}}:= \int_{\R^n} \frac{|u(x)|^{p-1}}{1+ |x|^{n+ps}} \, \d x < \infty.
\end{align*}
For every measurable function $u: \R^n \to \R,$  we define the semi-norm
\begin{align*}
    [u]^p_{V^{s,p}(\R^n)} :=   \int_{\R^n \times \R^n} \frac{|u(x)-u(y)|^p}{|x-y|^{n+ps}} \, \d x \d y.
\end{align*} For the boundary value of solutions, we consider the space of functions
\begin{align*}
    V^{s,p}(\Omega|\R^n) := \biggl\lbrace u:\R^n \to \R \; : \; u|_{\Omega} \in \L^p(\Omega) \; , \; \frac{u(x)-u(y)}{|x-y|^{n/p+s}} \in \L^p(\Omega \times \R^n)   \biggr\rbrace,
\end{align*}
with the norm 
\begin{align*}
    \|u\|^p_{V^{s,p}(\Omega|\R^n)} := \int_{\Omega} |u(x)|^p \, \d x + \int_{\Omega \times \R^n} \frac{|u(x)-u(y)|^p}{|x-y|^{n+ps}} \, \d x \d y.
\end{align*}
We denote $V^{s,p}_{\loc}(\Omega|\R^n)$ as functions $u: \R^n \to \R$ which satisfy
\begin{align*}
    \int_{K} |u(x)|^p \, \d x +  \int_{K \times \R^n} \frac{|u(x)-u(y)|^p}{|x-y|^{n+ps}} \, \d x \d y <  \infty,
\end{align*}
for every compact set $K \subset \Omega.$ \note{Note that by fractional Poincar\'{e}-Sobolev inequality, see Theorem \ref{thm:sobolevpoincare-ineq}, $[\,]_{V^{s,p}}$ is a norm on $C^{\infty}_0(\Omega)$. The completion of the space $C^{\infty}_0(\Omega)$ in $V^{s,p}(\Omega|\R^n)$ with respect to the norm $[\,]_{V^{s,p}(\R^n)}$} is denoted by $V^{s,p}_0(\Omega|\R^n)$, and the space of \note{functions} with boundary $g \in V^{s,p}(\Omega|\R^n)$ is defined by
\begin{align*}
    V^{s,p}_g(\Omega|\R^n) := \biggl \lbrace u \in V^{s,p}(\Omega|\R^n): u-g \in V^{s,p}_0(\Omega|\R^n)     \biggr \rbrace.
\end{align*}
We introduce $(V_0^{s,p}(\Omega|\R^n))^{\star}$ as the dual space of $V_0^{s,p}(\Omega|\R^n)$, and let $\|\,\|_{(V_0^{s,p}(\Omega|\R^n))^{\star}}$ denote the natural norm on $(V_0^{s,p}(\Omega|\R^n))^\star$. Define the Hölder dual of $p$ by $p'$, i.e.,  $\frac{1}{p'} + \frac{1}{p}=1.$ For every measurable $u: \R^n \to \R$ and $K \subset \R^n,$ the norm $\|\,\|_{C^{\alpha}(K)}$ is defined by 
\begin{align*}
  \|u\|_{C^{\alpha}(K)} := \sup_{(x,y) \in K \times K} \frac{|u(x)-u(y)|}{|x-y|^{\alpha}}.
\end{align*}
The set of critical points of a differentiable function $u$ is denoted by $N_u$. Let $D \subset \Omega$ be an open \note{subset}. For every $\beta \geq 1 $, we define $C^{2}_{\beta}(D)$ as the space of $C^2$ functions $u$ on $D$, satisfying 
\begin{align*}
    \sup_{x \in D \setminus N_u} \biggl ( \frac{\min\{\dist(x,N_u)^{\beta-1},1\}}{|\nabla u(x)|} + \frac{|D^2 u(x)|}{\dist(x,N_u)^{\beta-2}} \biggr ) < \infty.
\end{align*}
For example, $u(x) = |x-x_0|^{\beta} \in C^2_{\beta}(\Omega)$ for every $x_0 \in \R^n, \note{\beta \geq 2}$. Note that we need to use the space $C^2_{\beta}(D)$ to define the notion of viscosity solutions, see \cite{KKL}.

\subsection{Fractional Poincar\'{e}-Sobolev  inequality and fractional Sobolev embedding} \label{sec:fracsobol}
We bring the fractional Poincar\'{e}-Sobolev inequality.

\begin{thm}\label{thm:sobolevpoincare-ineq}
Let $\Omega \subset \R^n$ be an open bounded \note{subset}. Then, for every $u \in C^{\infty}_0(\Omega)$, we have 
\[
\begin{aligned}
    \| u \|_{\L^{p_s^\ast}(\Omega)} \leq C_1 [u]_{V^{s,p}(\R^n)}, \quad &\text{ if } ps < n, \\
    \| u\|_{\L^\infty(\Omega)} \leq C_2 
    [u]_{V^{s,p}(\R^n)}, \quad &\text{ if } ps > n,\\
    \| u\|_{\L^q(\Omega)} \leq C_3 
    [u]_{V^{s,p}(\R^n)}, \quad &\text{ if } ps = n,
\end{aligned}
\]
for every  $1\leq q< \infty$, where $C_1$ depends on $n,s,p$, $C_2$ depends on $n,s,p, \Omega$, and $C_3$ depends on $n,s,p,q, \Omega$.
\end{thm}

\note{\begin{proof}
    The proof is essentially contained in \cite[Thm 8.1, 8.2, 9.1]{P}. However, since the norms and spaces are slightly different, we bring the proof for the reader's convenience. First, by \cite[Prop 4.1 and 4.5]{BS}, we derive that the fractional Sobolev spaces in \cite{P}, which are defined via interpolation, are equivalent to $V^{s,p}(\R^n)$. Now, for the case $p s < n$, we refer to \cite[Thm 8.1]{P}. Moreover, the proof of the case $ps >n$ exists in \cite[Prop. 2.9]{BLP}. For last case $ps = n$, we first use \cite[Thm 9.1]{P} and H\"older's inequality to imply that 
\begin{align}
    \label{eq:nonhomogenousineq}
        \|u\|_{\L^q(\Omega)} \leq C (\|u\|_{\L^p(\R^n)} + [u]_{V^{s,p}(\R^n)}),
    \end{align}
    for every $u \in C^{\infty}_0(\Omega)$ and $1 \leq q < \infty$, where $C$ is a constant depending on $n,s,p,q,\Omega$. Now, by \cite[Lem. 2.4]{BLP}, it is obtained that 
    \begin{align}
    \label{eq:PoincareLp}
        \|u\|_{\L^p(\Omega)} \leq C' [u]_{V^{s,p}(\R^n)},
    \end{align}
for every $u \in C^{\infty}_0(\Omega),$ where $C'$ is a constant depending on $n,s,p, \Omega$. Hence, by \eqref{eq:nonhomogenousineq} and \eqref{eq:PoincareLp}, we finish the proof for case of $ps=n.$

\end{proof}}

\note{\begin{rem}
\label{rem:Dualfractionalsobolev}
    By Hölder's inequality and Theorem \ref{thm:sobolevpoincare-ineq}, we have $ \L^{(p_s^*)'}(\Omega)  \subset (V_0^{s,p}(\Omega|\R^n))^{\star}$ if $ps \neq n$, where we define the action of $f \in \L^{(p_s^*)'}(\Omega)$ on $u \in V_0^{s,p}(\Omega|\R^n)$ by the paring 
\begin{align*}
   \int_{\R^n} f\, u \, \d x. 
\end{align*}
 Moreover, by Hölder's inequality and Theorem \ref{thm:sobolevpoincare-ineq}, $\L^q(\Omega) \subset (V_0^{s,p}(\Omega|\R^n))^{\star}$ if $ps =n$ for every $1< q<\infty.$
\end{rem}}

 Note that, for every $1 \leq q < p_s^{\ast} $, we have 
$$\| u \|^p_{\L^q(\Omega)} \leq C [u]^p_{V^{s,p}(\R^n)}, \quad \text{for } u \in V_0^{s,p}(\Omega|\R^n),$$
where we used Theorem \ref{thm:sobolevpoincare-ineq} and H\"older's inequality. Hence, the eigenvalue $\Lambda_{p,q}$ defined in \eqref{eq:fist-eignevalue} is strictly positive, and the inverse $\Lambda_{p,q}^{-1}$ is the best constant $C$ in the above inequality.
\begin{thm}
\label{thm:fractionalsobolev}
   The space  $V^{s,p}_0(\Omega|\R^n)$ is compactly embedded in $\L^{q}(\Omega)$ for bounded open sets $\Omega \subset \R^n$ and $1<
   q\leq p.$
\end{thm}
\begin{proof}
   Similar to \cite[Thm. 2.2]{AA}, we consider a large enough ball $B \subset \R^n$ such that $\Omega \subset B$. Then, $V^{s,p}_0(\Omega|\R^n) \subset V^{s,p}_0(B|\R^n)$ and $B$ is an extension domain, see \cite[Thm. 5.4]{DPV}. By \cite[Thm. 7.1]{DPV}, $V^{s,p}_0(B|\R^n)$ is compactly embedded in $\L^q(B)$ for $1<q \leq p$. In conclusion, $V^{s,p}_0(\Omega|\R^n)$ is compactly embedded in $\L^q(\Omega)$ for $1<q \leq p$.
\end{proof}

\subsection{Notions of solutions}
In this subsection, we bring the different notions of solutions for fractional $p$-Laplacian equations.

The first notion is pointwise solutions. 

\begin{defn}
    For every $f \in C(\Omega)$, we say that $u \in \L^{p-1}_{ps}(\R^n) \note{\cap C(\Omega)}$ satisfies $(-\Delta_p)^s u =f$ pointwise in $\Omega$ if 
\begin{align*}
&   2 \mathrm{P.V.} \int_{\R^n} \frac{|u(x)-u(y)|^{p-2} (u(x)-u(y))}{|x-y|^{n+ps}} \, \d y\\&:=
 2\lim_{\epsilon \to 0}  \int_{\R^n \setminus B(x,\epsilon)} \frac{|u(x)-u(y)|^{p-2} (u(x)-u(y))}{|x-y|^{n+ps}} \, \d y \\&= f(x), \quad \text{for } x \in \Omega.
\end{align*}

\end{defn}
For simplicity, we remove the notation $\mathrm{P.V.}$ in the rest of the work.

The following is the definition of the locally weak and weak solutions.

\begin{defn}
    Let $f \in \L^1_{\loc}(\Omega).$ We say that $u \in \L_{ps}^{p-1}(\R^n) \cap V^{s,p}_{\loc}(\Omega|\R^n)$ is locally weak \note{subsolution (supersolution) of the equation $(-\Delta_p)^s u =  f $ in $\Omega$, or equivalently we say that $(-\Delta_p)^s u \leq (\geq)\, f$ in $\Omega$ in the locally weak sense,} if 
    \begin{align}
    \label{eq:weaksolambig}
  \int_{\R^n \times \R^n} \frac{|u(x)-u(y)|^{p-2} (u(x)-u(y))\, (\phi(x)-\phi(y)) }{|x-y|^{n+ps}} \, \d x \d y \leq (\geq) \int_{\Omega} f(x)\, \phi(x) \, \d x,
    \end{align}
for all non-negative $\phi \in C^{\infty}_0(\Omega).$ \note{ Note that, if $\phi$ is supported in $K \Subset \Omega$, we have 
\begin{align*}
       &\int_{K \times K} \frac{|u(x)-u(y)|^{p-2} (u(x)-u(y))\, (\phi(x)-\phi(y)) }{|x-y|^{n+ps}} \, \d x \d y \\&+  2 \int_{K \times \R^n} \frac{|u(x)-u(y)|^{p-2} (u(x)-u(y))\, (\phi(x)-\phi(y)) }{|x-y|^{n+ps}} \, \d x \d y \\&=\int_{\R^n \times \R^n} \frac{|u(x)-u(y)|^{p-2} (u(x)-u(y))\, (\phi(x)-\phi(y)) }{|x-y|^{n+ps}} \, \d x \d y,  
\end{align*}
 so the left-hand-side of \eqref{eq:weaksolambig} is well-defined for  $u \in \L_{ps}^{p-1}(\R^n) \cap V^{s,p}_{\loc}(\Omega|\R^n)$. If $u$ is a locally weak subsolution and supersolution of  $(-\Delta_p)^s u = f$ in $\Omega$}, then we say that $(-\Delta_p)^s u = f$ locally weakly in $\Omega$.
In the case that $u \in V_g^{s,p}(\Omega|\R^n)$ for a function $g \in V^{s,p}(\Omega|\R^n)$, $f \in (V_0^{s,p}(\Omega|\R^n))^{\star}$, and $(-\Delta_p)^s u = f$ locally weakly in $\Omega$, we call $u$ a weak solution of $(-\Delta_p)^s u =f$ in $\Omega.$
    
\end{defn}

The following is the definition of viscosity solutions.
\begin{defn}
\label{defn:viscosity}
    We say that $u \in \L^{p-1}_{ps}(\R^n) \cap C(\Omega)$ is a viscosity subsolution (supersolution) of $\note{(-\Delta_p)^s u = f}$ in $\Omega$ \note{for $f \in C(\Omega)$}, \note{or equivalently we say that $(-\Delta_p)^s u \leq (\geq)\, f$ in $\Omega$ in the viscosity sense,} if for every $\note{x_0} \in \Omega$, \note{$\overline{B(\note{x_0},r)} \subset \Omega$}, and a function $\phi \in C^2(B(\note{x_0},r))$ \note{which} touches $u$ from above (below) at $\note{x_0},$ i.e., 
    \begin{align*}
        \phi (\note{x_0}) &= u(\note{x_0})\\
        \phi &> (<) u, \quad \text{in } \overline{B(\note{x_0},r)} \setminus \{\note{x_0}\},
    \end{align*}
and \note{satisfies at least} one of the following: \\
$(a).  \nabla \phi(x_0) \neq 0$ or $p > \frac{2}{2-s},$\\
$(b).  \nabla \phi(x_0) =0$, $x_0$ is an isolated critical point of $\phi$,  $\phi \in C^{2}_{\beta}(B(x_0,r))$ for some $\beta > \frac{ps}{p-1}$, and $1<p \leq \frac{2}{2-s}$,\\
 we have 
\begin{align*}
    2\int_{\R^n} \frac{|w(\note{x_0})-w(y)|^{p-2} (w(\note{x_0})-w(y))}{|\note{x_0}-y|^{n+ps}} \, \d y \leq (\geq) f(\note{x_0}),
\end{align*}
where \begin{align*}
    w &= \phi, \quad \text{in } B(\note{x_0},r),\\
    w&= u, \quad \text{in } \R^n \setminus B(\note{x_0},r).
\end{align*}
    We define $(-\Delta_p)^s u = f$ in $\Omega$ in the viscosity sense if \note{$u$ is a viscosity subsolution and supersolution of $\note{(-\Delta_p)^s u = f}$ in $\Omega$.} 
\end{defn}

Now, we prove the weak supersolutions with right-hand sides are viscosity supersolutions, see \cite{KKL} for zero right-hand side. 

\begin{prop}
\label{prop:equivalenceweakvisco}
    Let $u \in \L^{p-1}_{ps}(\R^n) \cap C(\Omega) \cap V^{s,p}_{\loc}(\Omega|\R^n)$ be a locally weak subsolution (supersolution) of $ \note{(-\Delta_p)^s u = f}$ in $\Omega$, where $f \in C(\Omega).$ Then, \note{$u$ is a viscosity subsolution (supersolution) of $ (-\Delta_p)^s u = f$ in $\Omega$.}
\end{prop}
\begin{proof}
   We prove the case of subsolutions and the other one follows immediately by replacing $u$ with $-u.$ Assume that $x \in \Omega, r>0$ satisfy $\overline{B(x,r)} \subset \Omega$, and  $\phi \in C^2(B(x,r))$ touches $u$ from above at the point $x$, which satisfies either the condition $(a)$ or $(b)$ in Definition \ref{defn:viscosity}. Define \begin{align*}
    w &:= \phi, \quad \text{in } B(x,r),\\
    w&:= u, \quad \text{in } \R^n \setminus B(x,r).
\end{align*}
Assume that $(-\Delta_p)^s w(x) > f(x)$. Then, by \note{the} continuity of $f$ and $(-\Delta_p)^s w$ in $B(x,r)$, see \cite[Lem. 3.8]{KKL}, \note{we have} $(-\Delta_p)^s w \geq f+ \delta$ pointwise and locally weakly in $B(x,r')$ for some small enough $\delta>0,0< r'<r.$ By \cite[Lem. 3.9]{KKL}, for small enough $\epsilon>0$ and $\eta \in C^{2}_0(B(x,r'))$ satisfying $\eta(x)=1$ and $0 \leq \eta \leq 1$, we have $(-\Delta_p)^s (-w+\epsilon \eta) \leq -f $ pointwise in $B(x,r')$, \note{and $-w+\epsilon \eta$ is a locally weak subsolution of $(-\Delta_p)^s (-w+\epsilon \eta) = -f$ in $B(x,r')$}. Hence, $(-\Delta_p)^s (w-\epsilon \eta) \geq f $  pointwise, \note{and $w-\epsilon \eta$ is a locally weak supersolution of $(-\Delta_p)^s (w-\epsilon \eta) = f$ in $B(x,r')$}. Now, define another function 
\begin{align*}
    w' &:= \phi, \quad \text{in } B(x,r'),\\
    w'&:= u, \quad \text{in } \R^n \setminus B(x,r').
\end{align*}
Then,
\begin{align*}
   & 2\int_{\R^n } \frac{|(w'-\epsilon \eta)(z)-(w'-\epsilon \eta)(y)|^{p-2} ((w'-\epsilon \eta)(z)-(w'-\epsilon \eta)(y))}{|z-y|^{n+ps}} \, \d y \\ &\geq
   2 \int_{\R^n} \frac{|(w-\epsilon \eta)(z)-(w-\epsilon \eta)(y)|^{p-2} ((w-\epsilon \eta)(z)-(w-\epsilon \eta)(y))}{|z-y|^{n+ps}} \, \d y \\& = (-\Delta_p)^s (w-\epsilon \eta) \geq f(z),
\end{align*}
pointwise for $z \in B(x,r').$ Hence, $w'-\epsilon \eta = u$ on $\R^n \setminus B(x,r')$ and \note{and $w'-\epsilon \eta$ is a locally weak supersolution of $(-\Delta_p)^s (w'-\epsilon \eta) = f$ in $B(x,r')$}. Then, by \cite[Prop. 2.10]{IMS}, we have $w'-\epsilon \eta \geq u$ in $\R^n$. This is in contradiction with 
\note{\begin{align*}
w'(x)-\epsilon \eta(x) = \phi(x)-\epsilon \eta(x) = u(x) -\epsilon < u(x).    
\end{align*}} In conclusion, $(-\Delta_p)^s w(x) \leq f(x)$ for every $x \in \Omega$, which implies that $u$ is a viscosity subsolution of \note{$(-\Delta_p)^s u  = f$} in $\Omega$.
\end{proof}
The following proposition can be proved along the lines of \cite[Thm. 8]{LL} and \cite[Thm. 2.4]{KKL}.
\begin{prop}
\label{prop:existence}
    For every $f \in (V_0^{s,p}(\Omega|\R^n))^{\star}, g\in V^{s,p}(\Omega|\R^n)$, there exists a weak solution $u \in V^{s,p}_g(\Omega|\R^n)$ to
\begin{align*}
    (-\Delta_p)^s u = f,\quad \text{in } \Omega.
\end{align*}

\end{prop}

\begin{prop}
\label{prop:upperbound}
    Let $g \in V^{s,p}(\Omega|\R^n) \cap \L^{\infty}(\R^{n} \setminus \Omega)$ and $u \in V_g^{s,p}(\Omega|\R^n)$ be a locally weak subsolution of $ \note{(-\Delta_p)^s u = 1}$ in $\Omega.$ Then,

    \begin{align*}
        u \leq M, \quad \text{in } \Omega,
    \end{align*}
for a constant $M$ depending on $n,s,p,\Omega, \|g\|_{\L^{\infty}(\R^n \setminus \Omega)}.$
\end{prop}
\begin{proof}
    By Remark \ref{rem:Dualfractionalsobolev} and Proposition \ref{prop:existence}, there exists a weak solution $v \in V_0^{s,p}(\Omega|\R^n)$ of $(-\Delta_p)^s v =1$ in $\Omega$. Then, by the comparison principle, see \cite[Prop. 2.10]{IMS}, we have $v \geq 0$ and $$u \leq v + \|g\|_{\L^{\infty}(\R^n \setminus \Omega)}, \quad \text{in } \Omega.$$  This together with \cite[Thm. 3.1]{BP} concludes the proof.
\end{proof}

Finally, we prove the stability of viscosity solutions.

\begin{prop}
\label{prop:gammaconve}
    Let $u_i \in \L^{p-1}_{ps}(\R^n) \cap C(\Omega)$ be a uniformly bounded sequence of viscosity supersolutions of $\note{(-\Delta_p)^s u_i = f_i}$ in $\Omega$ for $f_i \in C(\Omega)$. Assume that $u_i$ converges locally uniformly in $\Omega$ to $u \in \L^{p-1}_{ps}(\R^n)$, $f_i$ converges locally uniformly in $\Omega$ to $f$, and $u_i$ converges pointwise a.e. in $\R^n$ to $u$. Then, \note{$u$ is a viscosity supersolutions of  $(-\Delta_p)^s u = f$} in $\Omega$.
\end{prop}
\begin{proof}
The proof follows the argument in  \cite[Lem. 4.5]{CS}. Let $x \in \Omega$ \note{and} $\phi \in C^2(\overline{B(x,r)})$ \note{touch} $u$ from below at $x$ and \note{satisfy} either $(a)$ or $(b)$ in Definition \ref{defn:viscosity}, where $\overline{B(x,r)} \subset \Omega$. Take a point $x_i \in \overline{B(x,r)}$ such that
    \begin{align*}
        u_i(x_i)-\phi(x_i) = \inf_{\overline{B(x,r)}} u_i -\phi.
    \end{align*}
    Since $x$ is the minimum point of $u-\phi$ in $\overline{B(x,r)}$ and $u_i$ converges uniformly to $u$ in $\overline{B(x,r)}$, the points $x_i$ converges to $x$ and $x_i$ is a local minimum for $u_i- \phi$ in $B(x,r)$. Define
\begin{equation*}
    w_i := \begin{cases}
       \phi + u_i(x_i) -\phi(x_i), \quad &\text{in } B(x,r),\\
       u_i, \quad &\text{in } \R^n \setminus B(x,r).
    \end{cases}
\end{equation*}
Then, $w_i$ touches $u_i$ from below at $x_i$ and
\begin{align*}
       2\int_{\R^n} \frac{|w_i(x_i)-w_i(y)|^{p-2} (w_i(x_i)-w_i(y))}{|x_i-y|^{n+ps}} \, \d y \geq f_i(x_i) 
\end{align*}
Now, Letting $i \to \infty$ and using Lebesgue dominated convergence theorem, we obtain
\begin{align*}
       2\int_{\R^n} \frac{|w(x)-w(y)|^{p-2} (w(x)-w(y))}{|x-y|^{n+ps}} \, \d y \geq f(x), 
\end{align*}
where 
\begin{equation*}
    w := \begin{cases}
       \phi, \quad &\text{in } B(x,r),\\
       u, \quad &\text{in } \R^n \setminus B(x,r).
    \end{cases}
\end{equation*}
This completes the proof.
\end{proof}

\section{Wiener criterion}
\label{sec:wiener}
In this section, we bring the proof of the Wiener criterion for weak solutions to the $s$-fractional $p$-Laplacian with a nonzero right-hand side. We mention that in \cite[Thm. 1.1]{KLL} the Wiener criterion for zero right-hand side is proved, and in \cite[A.4.]{Li} the sufficiency part of the results in \cite{KLL} has been extended to include equations with bounded right-hand sides. We derive the same result as \cite[A.4.]{Li} by using the case of zero right-hand side, \cite[Thm. 1.1]{KLL}, and a perturbation argument.

\begin{defn}
Define 
\begin{align*}
    \Cap_{s,p}(\overline{B(\xi_0,r)} \setminus \Omega,B(\xi_0,2r)) := \inf_v \iint_{\R^n \times \R^n} \frac{|v(x)-v(y)|^p}{|x-y|^{n+ps}} \, \d x \d y,
\end{align*}
    where the infimum is taken over all $v \in C^{\infty}_0(B(\xi_0,2r))$ such that $v \geq 1$ on $\overline{B(\xi_0,r)} \setminus \Omega$. We say that a point $\xi_0 \in \partial \Omega$ satisfies the Wiener criterion for the $s$-fractional $p$-Laplacian if
    \begin{align*}
        \int_{0}^1 \biggl(\frac{\Cap_{s,p}(\overline{B(\xi_0,r)} \setminus \Omega,B(\xi_0,2r))}{r^{n-ps}}\biggr)^{\frac{1}{p-1}} \, \frac{\d r}{r} = \infty.
    \end{align*}
\end{defn}

Now, we define the notion of regular boundaries and the Wiener criterion.
\begin{defn}
\label{def:regularbound}
A point $\xi_0 \in \partial \Omega$ is regular for the $s$-fractional $p$-Laplacian if for every $f \in \L^{\infty}(\Omega), g \in C(\R^n) \cap V^{s,p}(\Omega|\R^n),u \in V_g^{s,p}(\Omega|\R^n)$, satisfying $(-\Delta_p)^s u = f$ weakly in $\Omega$, we have $\lim_{\xi \to \xi_0} u(\xi) = g(\xi_0)$. We say that $\Omega$ has Wiener regular boundary for the $s$-fractional $p$-Laplacian if all the points on $\partial \Omega$ are regular for the $s$-fractional $p$-Laplacian.

\end{defn}

\begin{prop} \label{prop:Wiener}
    A point $\xi_0 \in \partial \Omega$ is regular for the $s$-fractional $p$-Laplacian if and only if it satisfies the Wiener criterion for the $s$-fractional $p$-Laplacian.
\end{prop}

\begin{proof}
   Our argument is similar to \cite[Lem. 29]{LL}. By \cite[Thm. 1.1]{KLL}, it follows that if $\xi_0$ is regular, then it satisfies the Wiener criterion for the $s$-fractional $p$-Laplacian. Now, assume that $\xi_0 \in \partial \Omega$ satisfies the Wiener criterion for the $s$-fractional $p$-Laplacian.  Let $f \in \L^{\infty}(\Omega)$ and $u \in V_g^{s,p}(\Omega|\R^n)$ be a weak solution of $(-\Delta_p)^s u=f$ with boundary value $g \in C(\R^n) \cap V^{s,p}(\Omega|\R^n)$. Assume that $B$ is large ball such that $\overline \Omega \subset B$, $2B$ is the ball with the same center as $B$ and double radius, \note{and $\eta\in C^{\infty}_0(\R^n \setminus B)$ is a non-negative function}, where $\eta =1$ on $\R^n \setminus 2B$. Then, by Lemma \cite[Lem. 2.8]{IMS}, $(-\Delta_p)^s(u+M \eta)=
    f + h_M$ weakly in $\Omega$, 
    where 
    \begin{align*}
        &h_M(x) \\&=2 \int_{\R^n \setminus B} \frac{|u(x)-u(y)-M\eta(y)|^{p-2} (u(x)-u(y)-M\eta(y)) - |u(x)-u(y)|^{p-2} (u(x)-u(y))}{|x-y|^{n+ps}} \, \d y,
    \end{align*}
for $M \in  \R$ and a.e. Lebesgue point $x \in \R^n$ of $u.$ Hence, for $M>0$ large enough, we have $$(-\Delta_p)^s(u-M\eta)\geq 0 \geq (-\Delta_p)^s(u+M \eta) \quad \text{\note{locally} weakly in }\Omega.$$ Now, by Proposition \ref{prop:existence}, there exist weak solutions $u_1 \in V^{s,p}_{g-M \eta}(\Omega|\R^n),u_2 \in V^{s,p}_{g+M \eta}(\Omega|\R^n)$, such that \begin{align*}
   (-\Delta_p)^s u_1 = (-\Delta_p)^s u_2 = 0, \quad \text{weakly in } \Omega.
\end{align*}
Then, by the comparison principle  \cite[Prop. 2.10]{IMS}, \note{we obtain}
\begin{align*}
    u_1 \leq u-M\eta, \, u_2 \geq u + M \eta, \quad \text{in } \Omega.
\end{align*}
In conclusion, by \note{the Wiener criterion for the $s$-fractional $p$-Laplacian} and \cite[Thm. 1.1]{KLL}, $u_1,u_2$ take their boundary values continuously and
\begin{align*}
\underset{{\xi \to \xi_0}}{\limsup}\, u(\xi) &= \underset{{\xi \to \xi_0}}{\limsup}\, u(\xi)+M \eta(\xi) \leq \lim_{\xi \to \xi_0} u_2(\xi)  =g(\xi_0) =\lim_{\xi \to \xi_0} u_1(\xi) \\&  \leq \underset{{\xi \to \xi_0}}{\liminf}\, u(\xi)-M \eta(\xi) = \underset{{\xi \to \xi_0}}{\liminf}\, u(\xi).  
\end{align*}
This completes the proof.

\end{proof}

\section{Hopf's lemma and Global boundary Harnack inequality}
\label{sec:HopfandHarnack}
In this section, we prove Hopf's lemma and global boundary Harnack inequality for solutions to fractional $p$-Laplacian equations.

First, we need the following two preliminary lemmas.

\begin{lem}
\label{lem:epsilon}
    Let $\delta>0, \beta >0$ and $u, v \in C(\overline \Omega_{\delta}).$ Assume that $0 \leq u \leq B v$ in $\R^n \setminus \Omega$ for a constant $B>0$ and 
\begin{align*}
    \sup_{\Omega} \frac{u}{C} -v >0
\end{align*}
for a constant $C>B$. Then, for every $\epsilon_0>0$, there exist $0<\epsilon<\epsilon_0, (x_{\epsilon},y_{\epsilon}) \in \Omega \times \Omega$ such that 
\begin{align}
\label{eq:maxpoints}
    \sup_{(x,y) \in \Omega_{\delta} \times \Omega_{\delta}} \frac{u(x)}{C}- v(y) - \frac{|x-y|^{\beta}}{\epsilon} =  \frac{u(x_{\epsilon})}{C}- v(y_{\epsilon}) - \frac{|x_{\epsilon}-y_{\epsilon}|^{\beta}}{\epsilon}. 
\end{align}
    
\end{lem}
\begin{proof}
    Let us define $(x_{\epsilon},y_{\epsilon})$ for every $\epsilon>0$, which satisfies \eqref{eq:maxpoints}. Then,
\begin{align}
\label{eq:lowerbound}
    \frac{u(x_{\epsilon})}{C}- v(y_{\epsilon}) - \frac{|x_{\epsilon}-y_{\epsilon}|^{\beta}}{\epsilon} \geq  \sup_{\Omega} \frac{u}{C} -v >0.
\end{align}
Hence,
\begin{align*}
    |x_{\epsilon}-y_{\epsilon}|^{\beta} \leq \epsilon \biggl(   \frac{u(x_{\epsilon})}{C}- v(y_{\epsilon})  \biggr).
\end{align*}
Letting $\epsilon \to 0$, we arrive at, up to a subsequence, $x_{\epsilon}, y_{\epsilon}$ converge to $z \in \overline \Omega_{\delta}.$ Hence, by \eqref{eq:lowerbound}, we get
\begin{align*}
\frac{u(z)}{C} - v(z) \geq  \sup_{\Omega} \frac{u}{C} -v >0.   
\end{align*}
In conclusion, \note{we obtain} $z \in \Omega$ since $C>B$ and $0 \leq u \leq B v$ in $\R^n \setminus \Omega.$
In particular, for small $0<\epsilon<\epsilon_0,$ we have $x_{\epsilon},y_{\epsilon} \in \Omega.$
\end{proof}

\begin{lem}
\label{lem:sequrep-laplacian}
    Let $r_0 \in \R, x_0 \in \R^n, y_0 \in \R^n$, and $\beta > \note{\max\big(\frac{ps}{p-1},2\big)}.$ Then, 
\begin{align*}
  \biggl| \int_{B(y_0,r_0) \setminus B(y_0,\epsilon)} \frac{||y-x_0|^{\beta}- |y_0-x_0|^{\beta} |^{p-2} (|y-x_0|^{\beta}- |y_0-x_0|^{\beta}) }{|y-y_0|^{n+ps}} \, \d y \biggr| \leq C_{r_0} ,
\end{align*}
for every $0<\epsilon< r_0$ such that $r_0 < \frac{|x_0-y_0|}{2}$ if $x_0 \neq y_0$, where  $C_{r_0}$ is independent of $y_0$ and satisfies $\lim_{r_0 \to 0} C_{r_0}=0.$

\end{lem}
\begin{proof}
    If $x_0 = y_0$, then
\begin{align*}
   & \biggl| \int_{B(x_0,r_0) \setminus B(x_0,\epsilon)} \frac{||y-x_0|^{\beta}- |x_0-x_0|^{\beta} |^{p-2} (|y-x_0|^{\beta}- |x_0-x_0|^{\beta}) }{|y-x_0|^{n+ps}} \, \d y \biggr|\\& =\note{\frac{n|B(0,1)|}{\beta(p-1)-ps}}  r_0^{\beta(p-1)-ps}. 
\end{align*}
    If $x_0 \neq y_0$, then the result follows by \note{\cite[Lem. 3.6, Lem. 3.7]{KKL}}.

\end{proof}
The following two lemmas are the main tools for the proof of Hopf's lemma, global boundary Harnack theorem, and the isolation of the first fractional $(s,p)$-eigenvalue. 

\begin{lem}
\label{thm:boundbelow}
   Let $\delta>0$ and $v_i \in \L^{p-1}_{ps}(\R^n) \cap C(\overline \Omega_{\delta})$ be a sequence of functions \note{which are viscosity supersolution of} $(-\Delta_p)^s v_i = f_i$ in $\Omega$ \note{for $f_i \in C(\Omega)$}. Assume that
$v_i$ converges uniformly to $v$ on compact subsets of $\Omega$, $v>0$ in $\Omega$, and
\begin{align*}
 \underset{i \to \infty}{\limsup}\, f_i(x_i)     \geq -2 \int_{K} \frac{v^{p-1}(y)}{|x-y|^{n+ps}} \, \d y, \quad &\text{if } \note{x_i \in \Omega,} \lim_{i \to \infty} x_i =x \in \partial \Omega, \, \underset{i \to \infty}{\limsup}\,  v_i(x_i) \leq 0,
\end{align*}
 where $K \Subset \Omega$. \note{Moreover, $u_i \in \L^{p-1}_{ps}(\R^n) \cap  C(\overline \Omega_{\delta})$ is a sequence of non-negative uniformly bounded functions which satisfy}\note{ 
\begin{align*}0 \leq u_i \leq B v_i, \quad &\text{in } \R^n \setminus \Omega,\\
(-\Delta_p)^s u_i \leq D, \quad &\text{in } \Omega,
 \end{align*}}
\note{in the viscosity sense} for every $i$, where $B,D$ are positive constants. \note{Then,}  
\begin{align*}
     u_i \leq C v_i, \quad \text{in } \R^n,
\end{align*}
for $i>N$, where $N,C$ are positive constants.

\end{lem}
\begin{proof}
    We prove the \note{lemma} in several steps. By taking $K \Subset \Omega$ a bit larger and using $v>0$, we assume that 
    \begin{align}
    \label{eq:upperlimit}
       \underset{i \to \infty}{\limsup}\, f_i(x_i)     > -2 \int_{K} \frac{v^{p-1}(y)}{|x-y|^{n+ps}} \, \d y,, \quad &\text{if } \lim_{i \to \infty} x_i =x \in \partial \Omega, \, \underset{i \to \infty}{\limsup}\, v_i(x_i) \leq 0.
    \end{align}
Assume that by contradiction, there is a sequence of $C_i>B$ increasing to infinity, such that
    \begin{align*}
       m_i := \sup_{\R^n} \frac{u_i}{C_i} -v_i>0.
    \end{align*}
    Since $C_i>B$ and $0 \leq u_i \leq B v_i$ in $\R^n \setminus \Omega,$ we have \note{$$m_i = \sup_{\Omega} \frac{u_i}{C_i} -v_i>0.$$}\note{Let us choose $\beta> \max\big(\frac{ps}{p-1},2\big).$} \note{Then, by} Lemma \ref{lem:epsilon}, there exist $0<\epsilon_i< \frac{1}{i(1+\|v_i\|_{\L^{\infty}(\Omega)})}$ and $ (x_i,y_i) \in \Omega \times \Omega$ such that 
    \begin{align}
    \label{eq:keypoints}
      0< m_i \leq \sup_{(x,y) \in \Omega_{\delta} \times \Omega_{\delta}} \frac{u_i(x)}{C_i}- v_i(y) - \frac{|x-y|^{\beta}}{\epsilon_i} =  \frac{u_i(x_{i})}{C_i}- v_i(y_{i}) - \frac{|x_{i}-y_{i}|^{\beta}}{\epsilon_i}.
    \end{align}  
 Without loss of generality, up to a subsequence, we assume $\underset{i \to \infty}{\lim}\, x_i = \Tilde{x}, \underset{i \to \infty}{\lim}\, y_i = \Tilde{y}.$ Then, by \eqref{eq:keypoints}, we get 
 \begin{align*}
    |\Tilde{x}- \Tilde{y}|^{\beta} = \underset{i \to \infty}{\limsup}\,  |x_i- y_i|^{\beta} \leq \underset{i \to \infty}{\limsup}\, \frac{\frac{u_i(x_{i})}{C_i}+ |v_i(y_{i})|}{i(1+\|v_i\|_{\L^{\infty}(\Omega)})} =0.   
 \end{align*}
 Hence, $\Tilde{x} = \Tilde{y} \in \overline{\Omega}$. If $\Tilde{x} = \Tilde{y} \in \Omega$, then we arrive at the following contradiction
 \begin{align*}
  0 \leq \underset{i \to \infty}{\liminf}\, m_i \leq  \underset{i \to \infty}{\liminf}\,  \frac{u_i(x_{i})}{C_i}- v_i(y_{i}) =- v(\Tilde{y})<0.
 \end{align*}
In conclusion, $\Tilde{x}= \Tilde{y} \in \partial \Omega$.
 Also, by \eqref{eq:keypoints} and uniform boundedness of the sequence $u_i$, we derive
 \begin{align*}
     0 \leq \underset{i \to \infty}{\liminf}\,  \frac{u_i(x_{i})}{C_i}- v_i(y_{i}) = \underset{i \to \infty}{\liminf} \,- v_i(y_{i}) = - \underset{i \to \infty}{\limsup}\, v_i(y_i).
 \end{align*}
Thus,
\begin{align}
    \label{eq:limsup}
    \underset{i \to \infty}{\limsup}\, v_i(y_i) \leq 0.
\end{align}
Taking $i$ large enough, we assume that $|x_i-y_i| \leq \frac{\delta}{3}$ and $x_i, y_i$ belong to $\Omega \setminus \overline{K}.$ Let $r_i>0$ be small enough, such that $B(x_i,r_i) \cup B(y_i,r_i) \subset \Omega \setminus \overline{K}$, $r_i < \frac{|x_i-y_i|}{2}$ if $x_i \neq y_i$, and
\begin{align}
    \label{eq:rcondition}
C_{r_i} \leq \frac{\epsilon_i^{p-1}}{i},   
\end{align}
where $C_{r_i}$ is the constant in Lemma \ref{lem:sequrep-laplacian}, \note{replacing} $x_0,y_0$ with $x_i,y_i$, respectively. Then, by \eqref{eq:keypoints}, the following inequalities are deduced:\\
$(i).$ If we set $x=x_i,$ then
\begin{align*}
    v_i(y) - v_i(y_i) \geq \frac{|y_i-x_i|^{\beta}}{\epsilon_i} - \frac{|y-x_i|^{\beta}}{\epsilon_i}, \quad \text{ in } \Omega_{\delta}.
\end{align*}
$(ii).$ If we set $x=y+ x_i-y_i$, then
\begin{align*}
       v_i(y) - v_i(y_i) \geq \frac{u_i(y+x_i-y_i)}{C_i}- \frac{u_i(x_i)}{C_i}, \quad \text{ in } \Omega_{\frac{\delta}{2}}.
\end{align*}
Notice that we used the fact that $|x_i -y_i| \leq \frac{\delta}{3}$ in the above inequality. \\
$(iii).$ If we set $y=y_i$, then
\begin{align*}
    \frac{u_i(x)}{C_i} - \frac{u_i(x_i)}{C_i} \leq \frac{|x-y_i|^{\beta}}{\epsilon_i} - \frac{|x_i-y_i|^{\beta}}{\epsilon_i}, \quad \text{in } \Omega_{\delta}.
\end{align*}
Now, define the functions
\begin{align*}
    w_i(y) &:= \begin{cases}
       v_i(y_i) + \frac{|y_i-x_i|^{\beta}}{\epsilon_i} - \frac{|y-x_i|^{\beta}}{\epsilon_i}, \quad &\text{in } B(y_i,r_i),\\
        v_i(y), \quad &\text{in } \R^n \setminus B(y_i,r_i).\end{cases}\\
        \Tilde{w}_i(x) &:= \begin{cases}
       \frac{u_i(x_i)}{C_i}+ \frac{|x-y_i|^{\beta}}{\epsilon_i} - \frac{|x_i-y_i|^{\beta}}{\epsilon_i}, \quad &\text{in } B(x_i,r_i),\\
        \frac{u_i(x)}{C_i}, \quad &\text{in } \R^n \setminus B(x_i,r_i).\end{cases}
\end{align*}
Hence, by $(i),(iii)$, $w_i$ touches $v_i$ from below at $y_i$, and $\Tilde{w}_i$ touches $\frac{u_i}{C_i}$ from above at $x_i.$ In conclusion,
\begin{equation}
\label{eq:viscosity}
\begin{aligned}
2\int_{\R^n}  \frac{|w_i(y)-w_i(y_i)|^{p-2} (w_i(y)-w_i(y_i))}{|y-y_i|^{n+ps}} \, \d y \leq -f_i(y_i),\\
2\int_{\R^n}  \frac{|\Tilde{w}_i(y)-\Tilde{w}_i(y_i)|^{p-2} (\Tilde{w}_i(y)-\Tilde{w}_i(y_i))}{|y-y_i|^{n+ps}} \, \d y \geq -\frac{D}{C_i^{p-1}}.
\end{aligned}
\end{equation}
Now, the rest of the proof aims at deriving a contradiction from the inequalities above and \eqref{eq:upperlimit}. By the first inequality in \eqref{eq:viscosity}, it is obtained that 
\begin{align*}
    &\underset{i \to \infty}{\liminf}\, -f_i(y_i)\\ &\geq  2\,\underset{i \to \infty}{\liminf}\, \int_{\R^n}  \frac{|w_i(y)-w_i(y_i)|^{p-2} (w_i(y)-w_i(y_i))}{|y-y_i|^{n+ps}} \, \d y\\
    &\geq 2\,\underset{i \to \infty}{\liminf}\, \int_{\R^n \setminus \Omega_{\frac{\delta}{2}}}  \frac{|v_i(y)-v_i(y_i)|^{p-2} (v_i(y)-v_i(y_i))}{|y-y_i|^{n+ps}} \, \d y \\& -2\, \underset{i \to \infty}{\limsup} \frac{1}{\epsilon_i^{p-1}} \int_{B(y_i,r_i)}  \frac{||y-x_i|^{\beta}-|y_i-x_i|^{\beta}|^{p-2} (|y-x_i|^{\beta}-|y_i-x_i|^{\beta})}{|y-y_i|^{n+ps}} \, \d y\\ &+2\,\underset{i \to \infty}{\liminf} \int_{K}  \frac{|v_i(y)-v_i(y_i)|^{p-2} (v_i(y)-v_i(y_i))}{|y-y_i|^{n+ps}} \, \d y\\&+2\,\underset{i \to \infty}{\liminf} \int_{\Omega_{\frac{\delta}{2}} \setminus (B(y_i,r_i) \cup K)}  \frac{|v_i(y)-v_i(y_i)|^{p-2} (v_i(y)-v_i(y_i))}{|y-y_i|^{n+ps}} \, \d y\\& \geq J_1+ J_2 + J_3+J_4.
\end{align*}
For the left-hand side, by \eqref{eq:upperlimit}, we obtain 
\begin{align*}
    \underset{i \to \infty}{\liminf}\, -f_i(y_i) = -\underset{i \to \infty}{\limsup}\, f_i(y_i) <  2 \int_{K} \frac{v^{p-1}(y)}{|y-\Tilde{y}|^{n+ps}} \, \d y.
\end{align*}
For the term $J_1$, we use $0 \leq u_i \leq B v_i$ in $\R^n \setminus \Omega$, \eqref{eq:limsup}, and Lebesgue dominated convergence to obtain
\begin{align*}
    J_1 \geq  2\,\underset{i \to \infty}{\liminf}\, \int_{\R^n \setminus \Omega_{\frac{\delta}{2}}}  \frac{-|v_i(y_i)|^{p-2} v_i(y_i)}{|y-y_i|^{n+ps}} \, \d y  \geq 0.
\end{align*}
By Lemma \ref{lem:sequrep-laplacian} and \eqref{eq:rcondition}, we get
\begin{align*}
    J_2 =0.
\end{align*}
For the term $J_3$, we use that $v_i$ converges uniformly to $v$ on $K$\note{, together with \eqref{eq:limsup},} to derive
\begin{align*}
    J_3 \geq 2\int_{K} \frac{v^{p-1}(y)}{|y-\Tilde{y}|^{n+ps}} \, \d y.
\end{align*}
The last term requires more work, and this is the place \note{where} we need to use the assumption $(-\Delta_p)^s u \leq D$ and the function $\Tilde{w}_i$. By $(ii)$ and integration by substitution, we obtain
\begin{align*}
    J_4 \geq 2\,\underset{i \to \infty}{\liminf}\, \int_{\Omega_{\frac{\delta}{2}}+ x_i -y_i \setminus (B(x_i,r_i) \cup K + x_i -y_i)} \frac{\biggl|\frac{u_i(y)}{C_i}-\frac{u_i(x_i)}{C_i}\biggr|^{p-2} \biggl(\frac{u_i(y)}{C_i}-\frac{u_i(x_i)}{C_i}\biggr)}{|y-x_i|^{n+ps}} \, \d y.
\end{align*}
Hence, by the second inequality in \eqref{eq:viscosity}, \eqref{eq:rcondition}, and Lemma \ref{lem:sequrep-laplacian}, \note{we imply} 
\begin{align*}
 0 &= \underset{i \to \infty}{\liminf}\, -\frac{D}{C_i^{p-1}} \\ &\leq 2\lim_{i \to \infty} \frac{1}{\epsilon_i^{p-1}} \int_{B(x_i,r_i)}  \frac{||y-y_i|^{\beta}-|x_i-y_i|^{\beta}|^{p-2} (|y-y_i|^{\beta}-|x_i-y_i|^{\beta})}{|y-x_i|^{n+ps}} \, \d y,\\ & + 2\lim_{i \to \infty}\int_{K+ x_i - y_i}  \frac{\biggl|\frac{u_i(y)}{C_i}-\frac{u_i(x_i)}{C_i}\biggr|^{p-2} \biggl(\frac{u_i(y)}{C_i}-\frac{u_i(x_i)}{C_i}\biggr)}{|y-x_i|^{n+ps}} \, \d y \\ &+ 2\lim_{i \to \infty}\int_{\R^n\setminus \Omega_{\frac{\delta}{2}}+ x_i - y_i} \frac{\biggl|\frac{u_i(y)}{C_i}-\frac{u_i(x_i)}{C_i}\biggr|^{p-2} \biggl(\frac{u_i(y)}{C_i}-\frac{u_i(x_i)}{C_i}\biggr)}{|y-x_i|^{n+ps}} \, \d y \\
 &+ 2\,\underset{i \to \infty}{\liminf}\,\int_{\Omega_{\frac{\delta}{2}} + x_i -y_i \setminus (B(x_i,r_i) \cup K+ x_i - y_i )} \frac{\biggl|\frac{u_i(y)}{C_i}-\frac{u_i(x_i)}{C_i}\biggr|^{p-2} \biggl(\frac{u_i(y)}{C_i}-\frac{u_i(x_i)}{C_i}\biggr)}{|y-x_i|^{n+ps}} \, \d y\\ &=  2\,\underset{i \to \infty}{\liminf}\,\int_{\Omega_{\frac{\delta}{2}} + x_i -y_i \setminus (B(x_i,r_i) \cup K+ x_i - y_i )} \frac{\biggl|\frac{u_i(y)}{C_i}-\frac{u_i(x_i)}{C_i}\biggr|^{p-2} \biggl(\frac{u_i(y)}{C_i}-\frac{u_i(x_i)}{C_i}\biggr)}{|y-x_i|^{n+ps}} \, \d y\\
 & \leq J_4.
\end{align*}
In conclusion,
$J_4 \geq 0$ and 
\begin{align*}
2\int_{K} \frac{v^{p-1}(y)}{|\Tilde{y}-y|^{n+ps}} \, \d y > J_1 + J_2 + J_3 + J_4 \geq 2\int_{K} \frac{v^{p-1}(y)}{|\Tilde{y}-y|^{n+ps}} \, \d y,
\end{align*}
which provides the contradiction.

\end{proof}

\begin{lem}
\label{lem:maxprinciple}
    Let $u \in \L^{p-1}_{ps}(\note{\R^n}) \cap C(\Omega)$ be a viscosity supersolution of $\note{(-\Delta_p)^s u = f}$ in $\Omega$ \note{for $f \in C(\Omega)$}. If there exists a point $x_0 \in \Omega$ such that $u(x_0) = \inf_{\R^n} u$, then
    \begin{align*}
         2\int_{\R^n} \frac{|u(x_0)-u(y)|^{p-2}(u(x_0)-u(y))}{|x_0-y|^{n+ps}} \, \d y \geq f(x_0).
    \end{align*}
\end{lem}
\begin{proof}
   \note{The proof} follows the \note{same} argument \note{as} \cite[Lem. 5.3]{AA}. Let $x_0 \in \Omega$ satisfy $u(x_0) = \inf_{\R^n} u.$ Since $u(x_0)-u(y) \leq 0$ for $y \in \R^n$, the integral above is well-defined without $\mathrm{P.V.}$ and might be $- \infty$. \note{Let us fix $\beta >\max\big( \frac{ps}{p-1},2\big)$.} Define $u_{\epsilon}(y)= u(x_0)-|x_0-y|^{\beta}$ for \note{$y \in B(x_0,\epsilon)$, and} $u_{\epsilon}= u$ in $\R^n \setminus B(x_0,\epsilon)$ for $\epsilon>0$ small enough, such that $\overline{B(x_0,\epsilon)} \subset \Omega$. \note{Then,  $u_{\epsilon}$} touches $u$ from below at $x_0$ and satisfies either the condition $(a)$ or $(b)$ in Definition \ref{defn:viscosity}. Hence,
    \begin{align*}
 &\note{\frac{2 n |B(0,1)|}{\beta(p-1)-ps}}\epsilon^{\beta(p-1)-ps}+ 2 \int_{\R^n
 \setminus B(x_0,\epsilon)}  \frac{|u(x_0)-u(y)|^{p-2}(u(x_0)-u(y))}{|x_0-y|^{n+ps}} \, \d y\\&= 2\int_{ B(x_0,\epsilon)}  \frac{|x_0-y|^{\beta (p-1)}}{|x_0-y|^{n+ps}}\, \d y+ 2\int_{\R^n \setminus B(x_0,\epsilon)}  \frac{|u_{\epsilon}(x_0)-u_{\epsilon}(y)|^{p-2}(u_{\epsilon}(x_0)-u_{\epsilon}(y))}{|x_0-y|^{n+ps}} \, \d y \\ &= 2\int_{\R^n} \frac{|u_{\epsilon}(x_0)-u_{\epsilon}(y)|^{p-2}(u_{\epsilon}(x_0)-u_{\epsilon}(y))}{|x_0-y|^{n+ps}}  \, \d y  \geq f(x_0).
\end{align*}
\note{In conclusion,
\begin{align*}
   f(x_0) + 2 \int_{\R^n
 \setminus B(x_0,\epsilon)}  \frac{|u(y)-u(x_0)|^{p-2}(u(y)-u(x_0))}{|x_0-y|^{n+ps}} \, \d y \leq \note{\frac{2 n |B(0,1)|}{\beta(p-1)-ps}}\epsilon^{\beta(p-1)-ps}.
\end{align*}}Now, letting $\epsilon \to 0$ and using the monotone convergence theorem, it is obtained that
\begin{align*}
 f(x_0) \leq  2\int_{\R^n} \frac{|u(x_0)-u(y)|^{p-2}(u(x_0)-u(y))}{|x_0-y|^{n+ps}} \, \d y.
\end{align*}
\end{proof} In particular, in the above lemma, if $u$ is non-negative and \note{$$f(x_0)> - 2\int_{\R^n} \frac{u^{p-1}(y)}{|x_0-y|^{n+ps}} \, \d y,$$}for any $x_0 \in \Omega$ satisfying $u(x_0)=0$, then $u>0$ in $\Omega$. This is the strong maximum principle for fractional $p$-Laplacian equations. We bring an example to justify the sharpness of the condition \note{$$f(x_0)> - 2\int_{\R^n} \frac{u^{p-1}(y)}{|x_0-y|^{n+ps}} \, \d y.$$} Let $\beta > \frac{ps}{p-1}$ \note{and} 
\begin{align*}
    u(x) := \begin{cases}
|x|^{\beta},\quad & x \in B(0,1)\\
1, \quad &x \in \R^n \setminus B(0,1).
    \end{cases}
\end{align*}
Then, $u(0)=0$ and
\begin{align*}
    (-\Delta_p)^s u(0)= -2\int_{\R^n} \frac{u^{p-1}(y)}{|y|^{n+ps}} \, \d y.
\end{align*}

Now, we prove Hopf's lemma.

\begin{proof}[Proof of Lemma \ref{lem:Hopf's lemma}]
By Lemma \ref{lem:maxprinciple} and the condition  $$f(x_0) > -2 \int_{\R^n} \frac{u^{p-1}(y)}{|x_0-y|^{n+ps}} \, \d y, \quad \text{if } x_0 \in \Omega, \, u(x_0)=0,$$ we obtain $u>0$ in $\Omega$. Finally, if we set $u_i =u_{\tor}, v_i = u$ in Lemma \ref{thm:boundbelow} and use 
\begin{align*}
      \,\underset{\Omega \ni x \to x_0}{\limsup}\, f(x) \geq -2\int_{K} \frac{u^{p-1}(y)}{|x_0 - y|^{n+ps}} \, \d y, \quad & \text{if } x_0 \in \partial \Omega,  u(x_0) =0,
\end{align*}
we conclude that $u \geq C u_{\tor}$ for a positive constant $C$.
\end{proof}

\begin{lem}
\label{lem:newmaximum}
    Let $\delta>0, u \in \L^{p-1}_{ps}(\R^n) \cap C(\overline \Omega_{\delta}) $ be a viscosity supersolution of 
    \begin{align*}
       \note{(-\Delta_p)^s u = -2(\diam \Omega)^{-(n+ps)} \int_{K} u^{p-1}(y) \, \d y, \quad \text{in } \Omega,} 
    \end{align*}
    for a \note{subset} $K \Subset \Omega,$ \note{which satisfies $u \geq 0$ in $\R^n \setminus \Omega.$} Then, either $u =0$ a.e. in $\R^n \setminus K$, or $u>0$ in $\Omega$.
\end{lem}
\begin{proof}
 If $u>0$ in $\Omega$, then the proof is complete. Now, assume that $u(x_0)=\inf_{\R^n}u \leq 0$ for a point $x_0 \in \Omega.$ Then, by Lemma \ref{lem:maxprinciple}, 
    \begin{align*}
        &  2\int_{\R^n} \frac{|u(x_0)-u(y)|^{p-2} (u(x_0)-u(y))}{|x_0-y|^{n+ps}} \, \d y \geq - 2 (\mathrm{diam}\, \Omega)^{-(n+ps)}\int_{K} u^{p-1}(y)\, \d y \\
 &\geq 2 (\mathrm{diam}\, \Omega)^{-(n+ps)}\int_{K} |u(x_0)- u(y)|^{p-2}(u(x_0)-u(y))\, \d y \\& \geq 2\int_{K} \frac{|u(x_0)- u(y)|^{p-2}(u(x_0)-u(y))}{|x_0-y|^{n+ps}} \, \d y. 
    \end{align*}
Hence, $u(y) = u(x_0) \leq 0$ a.e. in $\R^n \setminus K.$ In conclusion, by the assumption $u \geq 0$ in $\R^n \setminus \Omega$, we arrive at $u =0$ a.e. in $\R^n \setminus K.$
    
\end{proof}

\begin{proof}[Proof of Theorem \ref{thm:boundaryharnack}]
    Let $\delta>0, u_i \in C(\overline \Omega_{\delta}) \cap V_{g_{u_i}}^{s,p}(\Omega|\R^n), v_i \in C(\overline \Omega_{\delta}) \cap V_{g_{v_i}}^{s,p}(\Omega|\R^n)$ be positive on $\Omega$, 
\begin{align*}
     0 \leq \frac{1}{B} g_{v_i} \leq g_{u_i} \leq B g_{v_i} \leq M, \quad &\text{in } \R^n \setminus \Omega,
\end{align*}
and  \begin{equation}
    \label{eq:viscosityequation}
    \begin{aligned}
   - 2 (\mathrm{diam}\, \Omega)^{-(n+ps)}\int_{K} u_i^{p-1}(y)\, \d y \leq (-\Delta_p)^s u_i \leq  1, \quad & \text{in } \Omega,\\
    -2 (\mathrm{diam}\, \Omega)^{-(n+ps)}\int_{K} v_i^{p-1}(y) \, \d y  \leq   (-\Delta_p)^s v_i \leq 1, \quad &\text{in } \Omega,
    \end{aligned}
    \end{equation}
    in the locally weak sense, where $K \Subset \Omega$. We also assume that $$u_i(x_0) \geq D, \, v_i(x_0) \geq D \, (\|u_i\|_{\L^q(\Omega)} \geq D, \|v_i\|_{\L^q(\Omega)} \geq D),$$ for $D > 0, 1<q< \infty.$ \note{By} Proposition \ref{prop:upperbound}, $\|u_i\|_{\L^{\infty}(\Omega)} + \|v_i\|_{\L^{\infty}(\Omega)} \leq C$ for a constant $C$ depending on $n,s,p,\Omega,B, M.$ Also, by local Hölder regularity, see \cite[Thm. 5.4]{IMS}, the Arz\'ela-Ascoli theorem, and passing to a subsequence, $u_i,v_i$ converge uniformly to $u,v$, respectively, on compact subsets of $\Omega.$ Define $u=v=0$ in $\R^n \setminus \Omega$ and
    \begin{equation*}
    \begin{aligned}
    \Tilde{u}_i &:= \begin{cases}
        u_i, \quad \text{in } \Omega,\\
        0, \quad \text{in } \R^n \setminus \Omega,
    \end{cases}\\
     \Tilde{v}_i &:= \begin{cases}
        v_i, \quad \text{in } \Omega,\\
        0, \quad \text{in } \R^n \setminus \Omega.
        \end{cases}
        \end{aligned}
    \end{equation*} 
Then, by \eqref{eq:viscosityequation} and Definition \ref{defn:viscosity}, we have 
\begin{align*}
       - 2 (\mathrm{diam}\, \Omega)^{-(n+ps)}\int_{K} \Tilde{u}_i^{p-1}(y)\, \d y \leq (-\Delta_p)^s \Tilde{u}_i, \quad & \text{in } \Omega,\\
    -2 (\mathrm{diam}\, \Omega)^{-(n+ps)}\int_{K} \Tilde{v}_i^{p-1}(y) \, \d y  \leq   (-\Delta_p)^s \Tilde{v}_i, \quad &\text{in } \Omega,
\end{align*}
in the viscosity sense. Hence, by \note{the} Lebesgue dominated convergence, Proposition \ref{prop:equivalenceweakvisco}, and Proposition \ref{prop:gammaconve}, \note{we have}
    \begin{align*}
        u(x_0) \geq D, \, v(x_0) \geq D  \, (\|u\|_{\L^q(\Omega)} \geq D, \|v\|_{\L^q(\Omega)} \geq D),
    \end{align*}
    and  \begin{align*}
   - 2 (\mathrm{diam}\, \Omega)^{-(n+ps)}\int_{K} u^{p-1}(y)\, \d y \leq (-\Delta_p)^s u, \quad & \text{in } \Omega,\\
    -2 (\mathrm{diam}\, \Omega)^{-(n+ps)}\int_{K} v^{p-1}(y) \, \d y  \leq   (-\Delta_p)^s v, \quad &\text{in } \Omega,
    \end{align*}
    in the viscosity sense.
Hence, by Proposition \ref{prop:equivalenceweakvisco} and Lemma \ref{lem:newmaximum}, $u_i,v_i$ satisfy \eqref{eq:viscosityequation} in the viscosity sense and $u,v$ are strictly positive in $\Omega$. Finally, by Lemma \ref{thm:boundbelow}, $\frac{u_i}{v_i}$ and $\frac{v_i}{u_i}$ are uniformly bounded from below in $\Omega$. This completes the proof.

\end{proof}
In the case that $K$ is empty, we obtain the following result:

\begin{cor}
     Let $\delta>0$, $u \in C(\overline \Omega_{\delta}) \cap V_{g_u}^{s,p}(\Omega|\R^n),v \in C(\overline \Omega_{\delta}) \cap V_{g_v}^{s,p}(\Omega|\R^n)$ satisfy  
    \begin{align*}
    u>0 , \,v>0 \quad & \text{in }  \Omega,\\
   0 \leq \frac{1}{B} g_v \leq g_u \leq B g_v \leq M, \quad &\text{in } \R^n \setminus \Omega,
    \end{align*}
  for $B>0$, $M \geq 0$, and
    \begin{align*}
  0 \leq (-\Delta_p)^s u \leq  1, \quad & \text{in } \Omega,\\
    0  \leq   (-\Delta_p)^s v \leq 1, \quad &\text{ in } \Omega,
    \end{align*}
    in the locally weak sense. If either $u(x_0)\geq D , \; v(x_0)\geq D$ for a fixed point $x_0 \in \Omega$ or $\|u\|_{\L^q(\Omega)} \geq D , \;  \|v\|_{\L^q(\Omega)}\geq D$ for $D>0,1 \leq q < \infty$, then
    \begin{align*}
       C_1 \leq \frac{u}{v} \leq C_2, \quad \text{in } \Omega,
    \end{align*} where $C_1, C_2$ are positive constants depending on $\Omega, \delta, n, s,p, D,B,M, x_0$ or $q$.
\end{cor}

%\note{\begin{rem}
%\label{rem:pinfinitycase}
 %   If $\Omega$ has a $C^{1,1}$ boundary, then the proof of Theorem \ref{thm:boundaryharnack} can be generalized to derive the case of $q= \infty$. To see this, we follow the proof of Theorem \ref{thm:boundaryharnack} to derive uniform bounds the sequences $u_i, v_i$ and ${(-\Delta)}^su_i, {(-\Delta)}^s v_i$. Then, by $C^{1,1}$ condition on the boundary and \cite[Prop: 7.2]{R}, $\|u_i\|_{C^{s}(\overline{\Omega})},\|v_i\|_{C^{s}(\overline{\Omega})} $ are uniformly bounded. In conclusion, we can apply the Arzelà–Ascoli theorem to derive, up to a subsequence, both $v_i$ and $u_i$ converge uniformly on $\overline{\Omega}$ to two functions $v,u$, respectively. Hence, $u$ and $v$ are not identically zero in $\Omega \setminus K.$ The rest of the argument follows the same as proof of Theorem \ref{thm:boundaryharnack}.
    
%\end{rem}
%}

\section{Eigenvalue problem}
\label{sec:isolation}
In this section, $\Omega \subset \R^n$ is a bounded open set with Wiener regular boundary for the $s$-fractional $p$-Laplacian. Note that, as mentioned in Section \ref{sec:fracsobol}, $\Lambda_{p,q}>0$ in \eqref{eq:fist-eignevalue} by fractional Poincar\'e-Sobolev theorem, see Theorem \ref{thm:sobolevpoincare-ineq}. Now, it is aimed to prove Theorem \ref{thm:isolation}. We divide the proof into several steps. First, we prove the following global boundedness of weak solutions.
\begin{prop}
\label{prop:globalbound}
    Let $u \in V_0^{s,p}(\Omega|\R^n)$ be a \note{nonzero} weak solution of
\note{\begin{align*}(-\Delta_p)^s u = \lambda \|u\|^{p-q}_{\L^q(\Omega)} |u|^{q-2} u, \quad & \text{in }  \Omega,\end{align*}}where $1<q<p^{\ast}_s, \lambda>0$. Then,
\begin{align*}
    \|u\|_{\L^{\infty}(\Omega)} \leq C \note{ \lambda^{\theta}},
\end{align*}
where \note{$\theta$ is a positive constant depending on $n,s,p,q$ and} $C$ is a constant depending on $n,s,p,q,\Omega.$    
\end{prop}
\begin{proof}
   \note{By Theorem \ref{thm:sobolevpoincare-ineq} and H\"older's inequality, we have $u \in \L^q(\Omega)$.
   Since the problem is scale-invariant, without loss of generality, we can assume that $\|u\|_{\L^q(\Omega)}=1.$
   Now, we consider three cases. First, if $ps >n$, we can simply apply Theorem \ref{thm:sobolevpoincare-ineq}, together with the equation for $u,$ to derive that 
\begin{align*}
    \|u\|^p_{\L^{\infty}(\Omega)} \leq C [u]^p_{V^{s,p}(\R^n)}
    = C \lambda,
\end{align*}
which concludes the proposition for $\theta = \frac{1}{p}$. 
Now, for $ps = n$, we have 
\begin{align*}
    \|u\|_{\L^{\infty}(\Omega)} \leq C \| \, \lambda |u|^{q-2} u \, \|^{\frac{1}{p-1}}_{\L^{\frac{q}{q-1}}(\Omega)} = C \lambda^{\frac{1}{p-1}},
\end{align*}
where we used \cite[Lem. 2.3]{MPSY}. This completes the proposition for $\theta = \frac{1}{p-1}$. Finally, if $ps <n$, we use \cite[Lem. 2.3]{MPSY} multiple times until we get $\L^{\infty}$-estimate for $u$. We note that by Theorem \ref{thm:sobolevpoincare-ineq} and the equation for $u,$ we derive that
\begin{align*}
    \|u\|_{\L^{p_s^*}(\Omega)} \leq C [u]_{V^{s,p}(\R^n)} =  C \lambda^{\frac{1}{p}}.
\end{align*}
For the case that $u \in \L^r(\Omega)$ where $r > \frac{(q-1) n}{ps}$, we can use \cite[Lem. 2.3]{MPSY} once to imply that 
\begin{align*}
    \|u\|_{\L^{\infty}(\Omega)} \leq C \| \, \lambda |u|^{q-2} u \, \|^{\frac{1}{p-1}}_{\L^{\frac{r}{q-1}}(\Omega)} = C \lambda^{\frac{1}{p-1}} \|u\|^{\frac{q-1}{p-1}}_{\L^{r}(\Omega)},
\end{align*}
which concludes the proof. Now, if $u \in \L^{r}(\Omega)$ for $r = (q-1) \frac{n}{ps}$, we get 
\begin{align*}
    \|u\|_{\L^{t}(\Omega)} \leq |\Omega|^{\frac{1}{t}- \frac{1}{r}} \|u\|_{\L^r(\Omega)}
\end{align*}
by H\"older's inequality, where $q-1<t <(q-1) \frac{n}{ps}.$ Hence, by \cite[Lem. 2.3]{MPSY}, we obtain 
\begin{align*}
    \|u\|_{\L^{r'}(\Omega)} \leq C \| \, \lambda |u|^{q-2} u \, \|^{\frac{1}{p-1}}_{\L^{\frac{t}{q-1}}(\Omega)} \leq C \lambda^{\frac{1}{p-1}} |\Omega|^{ \frac{q-1}{p-1} (\frac{1}{t}- \frac{1}{r} )} \|u\|^{\frac{q-1}{p-1}}_{\L^r(\Omega)},
\end{align*}
where $r' = \frac{n (p-1) t}{ n(q-1)-pst }$. Note that as $t$ converges to $(q-1) \frac{n}{ps}$, $r'$ goes to infinity. Hence, by choosing $t$ close enough to $(q-1) \frac{n}{ps}$, we arrive at $u \in \L^{r'}(\Omega)$ for $r' >(q-1) \frac{n}{ps}$. In conclusion, the proof follows from the previous case. Finally, we claim that if $p_s^* < \frac{n(q-1)}{ps}$ and we start with $u \in \L^{r_1}(\Omega)$ for some $p_s^*\leq r_1 < \frac{n (q-1)}{ps}$, then $u \in \L^{r_2}(\Omega)$ with $\frac{r_2}{r_1} > \alpha>1$, where $\alpha$ depends on $n,s,p,q$, and 
\begin{align}
\label{eq:betterLpestimate}
    \|u\|_{\L^{r_2}(\Omega)} \leq C \lambda^{\frac{1}{p-1}} \|u\|^{\frac{q-1}{p-1}}_{\L^{r_1}(\Omega)},
\end{align}
where $C$ depends on $n,s,p,q,\Omega$. Note that by using the claim finitely many times, we get that $u$ belongs to $\L^{r}(\Omega)$ for some $r \geq \frac{n(q-1)}{ps}$ which, together with previous cases, completes the proof. To prove the claim, we use \cite[Lem. 2.3]{MPSY} as before to obtain \eqref{eq:betterLpestimate} for 
\begin{align*}
    r_2 := \frac{n (p-1) r_1}{ n (q-1)-ps r_1}.
\end{align*}
Then, by $r_1 \geq p_s^*$ and $q < p_s^*$, we deduce 
\begin{align*}
    \frac{r_2}{r_1} \geq \frac{n(p-1)}{n(q-1) - ps p_s^*} > \frac{n(p-1)}{n(p_s^*-1) - ps p_s^*} = 1,
\end{align*}
   which concludes the above claim for $\alpha = \frac{n(p-1)}{n(q-1) - ps p_s^*}.$} 
\end{proof}

\note{\begin{rem}
    The same proof can be applied to the general equation 
\note{\begin{align*}(-\Delta_p)^s u = f(u), \quad & \text{in }  \Omega,\end{align*}}where $f : \R \to \R$ satisfies $ |f(u)| \leq C_1 + C_2 |u|^{q-1}$ for $1<q<p^{\ast}_s$ and positive constants $C_1,C_2$, to derive 
\begin{align*}
    \|u\|_{\L^{\infty}(\Omega)} \leq  g\left(\|u\|_{\L^q(\Omega)}\right),
\end{align*}
where $g: \R \to \R$ is a non-negative function depending on $n,s,p,q,\Omega.$ Moreover, we do not need the Wiener regular property of $\Omega$ in the proof of Proposition \ref{prop:globalbound}.
\end{rem}}

Second, we demonstrate that the first eigenvalue $\Lambda_{p,q}$ is simple for $1<q \leq p$, and the first eigenfunction does not change sign.
\begin{prop}
\label{prop:comparison}
For every $1<q \leq p$, the weak solutions $u \in V^{s,p}_0(\Omega|\R^n)$ of $$(-\Delta_p)^s u = \Lambda_{p,q} |u|^{q-2} u, \quad \text{in } \Omega,$$ satisfying $\|u\|_{\L^q(\Omega)}=1$, are proportional and \note{strictly positive or negative on $\Omega$.} 
\end{prop}
\begin{proof}
Assume that $u,v \in V^{s,p}_0(\Omega|\R^n)$ satisfy $\|u\|_{\L^q(\Omega)}=1=  \|v\|_{\L^q(\Omega)}$ and are weak solutions of 
\begin{align}
    \label{eq:weakequations}
(-\Delta_p)^s u = \Lambda_{p,q} |u|^{q-2} u,\, (-\Delta_p)^s v = \Lambda_{p,q} |v|^{q-2} v \quad \text{in } \Omega.    
\end{align}
 Then, by \note{the} triangle inequality,
\begin{align*}
    \iint_{\R^n \times \R^n} \frac{|\,|u(x)|-|u(y)|\,|^p }{|x-y|^{n+ps}} \, \d x \d y  \leq  \iint_{\R^n \times \R^n} \frac{|u(x)-u(y)|^p }{|x-y|^{n+ps}} \, \d x \d y,
\end{align*}
and the equality holds if and only if $u$ does not change sign. Hence, by \eqref{eq:fist-eignevalue}, the equality occurs above, \note{and, without loss of generality, we can assume that $u$ is non-negative on $\Omega$. To prove that $u$ is positive in $\Omega,$ we note that, by Proposition \ref{prop:Wiener} and Proposition \ref{prop:globalbound}, $u \in C(\overline{\Omega})$. In conclusion, by Corollary \ref{cor:key} and Proposition \ref{prop:equivalenceweakvisco}, we derive that $u>0$ in $\Omega.$ Likewise, up to a multiplication with $-1$, we can assume that $v >0$ in $\Omega.$

Now, we show that $u$ and $v$ are proportional.} Define $\chi_t := (t^{\frac{1}{q}} u, (1-t)^{\frac{1}{q}} v)$ and $\|\,\|_{l^q}$ as the $l^q$-norm in $\R^2.$ Since $t \to t^{\frac{p}{q}}$ is a convex function on $\R^+$, we have
\begin{align*}
   \|\chi_t(x) -\chi_t(y)\|^{p}_{l^q} \leq t |u(x)-u(y)|^p + (1-t) |v(x)-v(y)|^p, \quad \text{for } x,y \in \R^n.
\end{align*}
Hence, 
\begin{equation}
\label{eq:inequalities}
\begin{aligned}
    \int_{\R^n} \int_{\R^n} \frac{| \|\chi_t(x)\|_{l^q} - \|\chi_t(y)\|_{l^q}|^p}{|x-y|^{n+ps}} \, \d x \d y &\leq  t \int_{\R^n} \int_{\R^n} \frac{| u(x) - u(y)|^p}{|x-y|^{n+ps}} \, \d x \d y\\& + (1-t) \int_{\R^n} \int_{\R^n} \frac{| v(x) - v(y)|^p}{|x-y|^{n+ps}} \, \d x \d y \leq \Lambda_{p,q},
\end{aligned}
\end{equation}
by \note{the} triangle inequality and \eqref{eq:weakequations}. Also, \note{we have} $$\|\|\chi_t\|_{l^q}\|^q_{\L^q(\Omega)} = t \|u\|^q_{\L^q(\Omega)} + (1-t) \|v\|^q_{\L^q(\Omega)}  = 1.$$ Hence, by \eqref{eq:fist-eignevalue}, the inequalities in \eqref{eq:inequalities} are equalities. In conclusion, we have the equality 
\begin{align*}
       | \|\chi_t(x)\|_{l^q} -\|\chi_t(y)\|_{l^q}|  = \|\chi_t(x) -\chi_t(y)\|_{l^q} \quad \text{for a.e. } x,y \in \R^n
\end{align*}
in the triangle inequality. It follows that $\chi_t(x)= c(x,y) \chi_t(y)$ for a.e. $x,y \in \R^n.$ Hence, \note{we have} $\frac{u(x)}{v(x)} = \frac{u(y)}{v(y)}$ for a.e. $x,y \in \Omega.$

\end{proof}

The next proposition shows that only for the first eigenvalue there exists a non-negative eigenfunction.

\begin{prop}
    \label{prop:signchanginh}
Let $u \in V_0^{s,p}(\Omega|\R^n)$ be a weak solution of $(-\Delta_p)^s u = \lambda |u|^{q-1} u$ where $1<q\leq p$, $\lambda>0$, and $\|u\|_{\L^q(\Omega)}=1$. If $u$ is non-negative in $\Omega$, then $\lambda=\Lambda_{p,q}.$
\end{prop}
\begin{proof}
    The argument is the same as \cite[Thm. 4.1]{FP}. Let $\lambda>0$ and $u \in V_0^{s,p}(\Omega|\R^n)$ be a non-negative weak solution of $(-\Delta_p)^s u = \Lambda_{p,q} u^{q-1}$, where $\|u\|_{\L^q(\Omega)}=1$. Assume that $v \in V_0^{s,p}(\Omega|\R^n)$ is a non-negative weak solution of $(-\Delta_p)^s v = \lambda v^{q-1}$, satisfying $\|v\|_{\L^q(\Omega)}= 1$. Define the functions $u_{\epsilon} := u + \epsilon, v_{\epsilon} := v+ \epsilon,$ and 
\begin{align*}
    \sigma^{\epsilon}_{t}(x) := \biggl ( t u^q_{\epsilon}(x) + (1-t) v^q_{\epsilon}(x) \biggr )^{\frac{1}{q}}\, \quad \text{for } x \in \Omega, t\in [0,1]. 
\end{align*}
Since $t \to t^{\frac{p}{q}}$ is a convex function on $\R^+$, \note{by the triangle inequality for $\|\,\|_{l^q}$}, it is obtained that
\begin{equation}
\label{eq:1}
\begin{aligned}
   & \int_{\R^n} \int_{\R^n} \frac{|\sigma_t^{\epsilon}(x) - \sigma_t^{\epsilon}(y)|^p}{|x-y|^{n+ps}} \, \d x \d y -  \int_{\R^n} \int_{\R^n} \frac{|v(x) - v(y)|^p}{|x-y|^{n+ps}} \, \d x \d y \\
   & \leq t \biggl( \int_{\R^n} \int_{\R^n} \frac{|u(x) - u(y)|^p}{|x-y|^{n+ps}} \, \d x \d y -  \int_{\R^n} \int_{\R^n} \frac{|v(x) - v(y)|^p}{|x-y|^{n+ps}} \, \d x \d y \biggr) \\
   & \leq t ( \Lambda_{p,q}- \lambda).
\end{aligned}
\end{equation}
Now, by \note{the} convexity of the map, $t \to t^{p}$, we have
\begin{equation}
\label{eq:2}
\begin{aligned}
   & \int_{\R^n} \int_{\R^n} \frac{|\sigma_t^{\epsilon}(x) - \sigma_t^{\epsilon}(y)|^p}{|x-y|^{n+ps}} \, \d x \d y -  \int_{\R^n} \int_{\R^n} \frac{|v(x) - v(y)|^p}{|x-y|^{n+ps}} \, \d x \d y \\ & \geq \note{p} \int_{\R^n} \int_{\R^n} \frac{|v(x) - v(y)|^{p-2}(v(x) - v(y))}{|x-y|^{n+ps}} (\sigma_t^{\epsilon}(x)-\note{\sigma^{\epsilon}_t(y)} - (v(x)-v(y))) \, \d x \d y.
\end{aligned}
\end{equation}
Hence, using $(-\Delta_p)^s v = \lambda v^{q-1} $ weakly in $\Omega$ and the test function $\sigma_{t}^{\epsilon}-v_{\epsilon}$, we get
\begin{equation}
\label{eq:3}
\begin{aligned}
     &\int_{\R^n} \int_{\R^n} \frac{|v(x) - v(y)|^{p-2}(v(x) - v(y))}{|x-y|^{n+ps}} (\sigma_t^{\epsilon}(x)-\note{\sigma^{\epsilon}_t(y)} - (v_{\epsilon}(x)-v_{\epsilon}(y))) \, \d x \d y \\ &=  \int_{\Omega} \lambda v^{q-1}(x) \, (\sigma_t^{\epsilon}(x)-v_{\epsilon}(x)) \, \d x.
\end{aligned}
    \end{equation}
In conclusion, by \eqref{eq:1}, \eqref{eq:2}, and \eqref{eq:3}, we arrive at 
\begin{equation}
\label{eq:4}
   \lambda \note{p} \int_{\Omega} v^{q-1}(x)  \frac{\sigma_t^{\epsilon}(x)-v_{\epsilon}(x)}{t} \, \d x \leq \Lambda_{p,q} - \lambda,
\end{equation}
where we used $v_{\epsilon}(x)- v_{\epsilon}(y) = v(x)-v(y)$ for all $x,y \in \R^n.$ Since $t \to t^{\frac{1}{q}}$ is a concave function of $t \in \R^+$, we have
$$v^{q-1}  \frac{\sigma_t^{\epsilon}-v_{\epsilon}}{t} \geq v^{q-1} (u-v), \quad \text{in } \Omega,$$
and $v^{q-1} (u-v) \in \L^1(\Omega)$ by Proposition \ref{prop:globalbound}. Then, by applying Fatou's lemma to \eqref{eq:4}, it is implied that
\begin{align*}
     \frac{\lambda \note{p}}{q} \int_{\Omega} \biggl(\frac{v(x)}{v_{\epsilon}(x)}\biggr)^{q-1} (u^{q}_{\epsilon}(x)-v^{q}_{\epsilon}(x) ) \leq \lambda \note{p} \, \underset{t \to 0}{\liminf}  \int_{\Omega} v^{q-1}(x)  \frac{\sigma_t^{\epsilon}(x)-v_{\epsilon}(x)}{t} \, \d x \leq \Lambda_{p,q} - \lambda,
\end{align*}
for small enough $\epsilon>0.$ Hence, by the assumption $v>0$ in $\Omega$ and Lebesgue dominated convergence, we deduce
\begin{align*}
    0 = \frac{\lambda \note{p}}{q} \int_{\Omega} u^q(x) - v^q(x) \, \d x \leq \Lambda_{p,q} - \lambda.
\end{align*}
Combining the above inequality with \eqref{eq:fist-eignevalue} implies \note{that} $\lambda = \Lambda_{p,q}.$

\end{proof}

\begin{proof}[Proof of Theorem \ref{thm:isolation}]
Let $\note{\lambda_i >\Lambda_{p,q}}, v_i \in V^{s,p}_0(\Omega|\R^n)$ be a sequence such that 
\begin{align*}
\|v_i\|_{\L^q(\Omega)}&=1,\\
\lim_{i \to \infty} \lambda_i &= \Lambda_{p,q},\\
(-\Delta_p)^s v_i &= \lambda_i |v_i|^{q-2}v_i,
\end{align*}
weakly in $\Omega.$  By Proposition \ref{prop:globalbound}, 
\begin{align}
\label{eq:Linfinitybound}
    \|v_i\|_{\L^{\infty}(\Omega)} \leq C \note{\lambda_i^{\theta} }
\end{align}
for \note{some} positive constants $\theta,C$, \note{which are independent of $i$}. Also, $[v_i]^p_{V^{s,p}(\R^n)} = \lambda_i$. Hence, by Theorem \ref{thm:fractionalsobolev}, \note{\cite[Thm. 2.7]{BLP}}, \note{and passing to a subsequence}, $v_i$ converges strongly in $\Omega$ to $u \in V^{s,p}_0(\Omega|\R^n)$ with respect to \note{the $\L^q$-norm} and $\|u\|_{\L^q(\Omega)}=1$. \note{Then}, 
\begin{align*}
    (-\Delta_p)^s u &= \Lambda_{p,q} |u|^{q-2}u, \quad \text{weakly in } \Omega.
\end{align*}
Now, by Proposition \ref{prop:comparison}, $u>0$ a.e. in $\Omega$ up to multiplication by a constant. By \eqref{eq:Linfinitybound}, Hölder's regularity for the $s$-fractional Laplacian, see \cite[Thm. 5.4]{IMS}, and \note{the} Arz\'ela-Ascoli theorem, $v_i$ converges uniformly to $u$ on compact subsets of $\Omega$ up to a subsequence and multiplication by a constant. Since $\Omega$ has a Wiener regular boundary for the $s$-fractional $p$-Laplacian, we derive $u,v_i, u_{\tor}$ belong to $C(\R^n)$. Now, we want to apply Lemma \ref{thm:boundbelow} to sequences $u_i =u,v_i.$ Since $u>0$, we need to only check that $\lim_{i \to \infty} \lambda_i |v_i(x_i)|^{q-2} v_i(x_i) = 0$ for every sequence $x_i \in \Omega$ such that $\lim_{i \to \infty} x_i \in \partial \Omega.$ To prove this, define \note{$$\Tilde{v}_i := \biggl(\lambda_i \|v_i\|_{\L^{\infty}(\Omega)}^{q-1}\biggr)^{\frac{1}{p-1}}  u_{\tor}.$$} Then, $\Tilde{v_i} = v_i=0$ in $\R^n \setminus \Omega$ and 
\begin{align*}
   (-\Delta_p)^s \Tilde{v}_i = \lambda_i \|v_i\|^{q-1}_{\L^{\infty}(\Omega)},
\end{align*}
weakly in $\Omega$. Hence, $- \Tilde{v}_i \leq v_i \leq \Tilde{v}_i$ in $\Omega$ by comparison principle, see \cite[Prop. 2.10]{IMS}. In particular, by \eqref{eq:Linfinitybound} and $u_{\tor} \in C(\R^n)$, we obtain \begin{align*}\lim_{i \to \infty} \lambda_i |v_i(x_i)|^{q-2} v_i(x_i)  &= \lim_{i \to \infty} \lambda_i |\Tilde{v}_i(x_i)|^{q-2} \Tilde{v}_i(x_i)\\&=\lim_{i \to \infty} \biggl(\lambda_i \|v_i\|_{\L^{\infty}(\Omega)}^{q-1}\biggr)^{\frac{q-1}{p-1}} |u_{\tor}(x_i)|^{q-2} u_{\tor}(x_i)  =0,\end{align*} for every sequence $x_i \in \Omega$ satisfying $\lim_{i \to \infty} x_i \in \partial \Omega$. In conclusion, by Lemma \ref{thm:boundbelow}, we obtain $v_i \geq C u$ in $\Omega$ for a constant $C>0$ and large enough $i$. In particular, $v_i>0$ in $\Omega$ for large enough $i$, which implies that $\lambda_i = \Lambda_{p,q}$ by Proposition \ref{prop:signchanginh}. This derives the desired contradiction.

\end{proof}

\begin{rem}
\label{rem:shortproof}
If $\Omega$ has a $C^{1,1}$ boundary \note{and $p \geq 2$}, then one can simplify the proof of Theorem \ref{thm:isolation} by using \cite[Thm. 1.1]{IMS1}. Indeed, taking a sequence of functions $v_i$ as in the above proof, we have, by \cite[Thm. 1.1]{IMS1}, $\frac{v_i}{\dist(x,\partial \Omega)^s}$ 
is uniformly bounded, and $$\biggl\|\frac{v_i}{\dist(x,\partial \Omega)^s}\biggr\|_{C^{\alpha}(\bar \Omega)} \leq C$$ for some positive constants $C, \alpha.$ By the same argument as proof of Theorem \ref{thm:isolation}, up to multiplication by \note{a} constant and \note{passing to a} subsequence, we derive \note{that} $v_i$ is sign-changing on $\Omega$ and converges pointwise to a non-negative function $u$. Hence, by the \note{Arz\'ela-Ascoli} theorem, $\frac{v_i}{\dist(x,\partial \Omega)^s}$ converges uniformly to $\frac{u}{\dist(x,\partial \Omega)^s}$ up to a subsequence. Since $\frac{v_i}{\dist(x,\partial \Omega)^s}$ is also continuous and sign-changing on $\Omega$, we derive \note{that} $\frac{u}{\dist(x,\partial \Omega)^s}$ goes to zero at a boundary point. This contradicts Corollary \ref{cor:key} and \cite[Lem. 2.3]{IMS1}, since $u \geq C_1 u_{\tor}(x) \geq C_2  \dist(x,\partial \Omega)^s$ for $x \in \Omega$, where $C_1,C_2$ are positive constants.

\end{rem}

\section{acknowledgement}

\note{The author wants to thank
Alireza Tavakoli for many motivating discussions and comments on this work.}

\def\cprime{$'$} \def\cprime{$'$} \def\cprime{$'$}

\end{document}